\documentclass{amsart}
\usepackage[english]{babel}
\usepackage[latin1]{inputenc}
\usepackage[dvips,final]{graphics}
\usepackage{amsmath,amsfonts,amssymb,amsthm,amscd,array,stmaryrd,mathrsfs, mathdots, epigraph}
\usepackage{pstricks}
 \usepackage[all]{xy}
 \usepackage{url}
\usepackage{multirow, blkarray}
\usepackage{booktabs}
\usepackage{textcomp}
 \usepackage[final]{epsfig}
\let\ssection=\section
\renewcommand{\section}{\setcounter{equation}{0}\ssection}

\newtheorem{lemma}{Lemma}[subsection]
\newtheorem{proposition}[lemma]{Proposition}
\newtheorem{theorem}[lemma]{Theorem}
\newtheorem{corollary}[lemma]{Corollary}

\theoremstyle{definition}
\newtheorem{example}[lemma]{Example}
\newtheorem{definition}[lemma]{Definition}
\newtheorem{notation}[lemma]{Notation}
\newtheorem{remark}[lemma]{Remark}

\newcommand{\s}{{\sigma}}

\newcommand{\proofend}{$\Box$\bigskip}

\newcommand{\R}{{\mathbb R}}
\newcommand{\T}{{\mathbb T}}
\newcommand{\Z}{{\mathbb Z}}

\newcommand{\RP}{{\mathbb {RP}}}  
\newcommand{\PP}{{\mathbb {P}}} 
\newcommand{\SL}{{\mathrm {SL}}}  
\newcommand{\GL}{{\mathrm {GL}}}
\newcommand{\Mat}{{\mathrm {Mat}}}

\newcommand{\Gr}{{\mathrm {Gr}}}  

\newcommand{\Gc}{{\mathcal G}}
\newcommand{\Cc}{{\mathcal C}}
\newcommand{\Ec}{{\mathcal E}}
\newcommand{\Fc}{{\mathcal F}}
\newcommand{\Sc}{{\mathcal S}}
\newcommand{\bfi}{{\mathbf {i}}}
\newcommand{\bft}{{\mathbf {t}}}

\begin{document}

\title[Difference equations, frieze patterns, combinatorial Gale transform]{
Linear difference equations, frieze patterns and combinatorial Gale transform}

\keywords{linear difference equations, frieze patterns, moduli spaces of $n$-gons,
rational periodic maps, Gale transform}

\author[S. Morier-Genoud]{Sophie Morier-Genoud}

\author[V. Ovsienko]{Valentin Ovsienko}

\author[R.E. Schwartz]{Richard Evan Schwartz}

\author[S. Tabachnikov]{Serge Tabachnikov}

\address{Sophie Morier-Genoud,
Institut de Math\'ematiques de Jussieu
UMR 7586
Universit\'e Pierre et Marie Curie
4, place Jussieu, case 247
75252 Paris Cedex 05
}

\address{
Valentin Ovsienko,
CNRS,
Institut Camille Jordan,
Universit\'e Claude Bernard Lyon~1,
43 boulevard du 11 novembre 1918,
69622 Villeurbanne cedex,
France}

\address{
Richard Evan Schwartz,
Department of Mathematics, 
Brown University, 
Providence, RI 02912, USA
}

\address{
Serge Tabachnikov,
Pennsylvania State University,
Department of Mathematics,
University Park, PA 16802, USA,
and 
ICERM, Brown University, Box1995,
Providence, RI 02912, USA
}

\email{sophiemg@math.jussieu.fr,
ovsienko@math.univ-lyon1.fr,
res@math.brown.edu,
tabachni@math.psu.edu
}

\date{}

\subjclass{}

\maketitle

\hfill{ \it To the memory of Andrei Zelevinsky}

\begin{abstract}
We study the space of linear difference equations with periodic
coefficients and (anti)periodic solutions.
We show that this space is isomorphic to the space of
tame frieze patterns and closely related to the moduli space of configurations
of points in the projective space.
We define the notion of combinatorial Gale transform which is a duality
between periodic difference equations of different orders. 
We describe periodic rational maps generalizing the classical Gauss map.
\end{abstract}

\tableofcontents

\section{Introduction}

Linear difference equations appear in many fields of mathematics, and
relate fundamental objects of geometry, algebra and analysis.
In this paper, we study the space of linear difference equations with periodic coefficients
and (anti)periodic solutions.
(The solutions of the equation are periodic if the order of the equation is odd,
and antiperiodic if the order is even.)
The space of such equations is a very interesting algebraic variety.
Despite the fact that this subject is  very old and classical,
this space has not been studied in  detail.

We prove that  the space ${\mathcal E\/}_{k+1,n}$
 of $n$-periodic linear difference equations
of order $k+1$ is equivalent to the space ${\mathcal F\/}_{k+1,n}$ of
tame $\SL_{k+1}$-frieze patterns of width $w=n-k-2$.  (Tameness is
the non-vanishing condition on certain determinants; See Definition \ref{tamedef}.)
We also show that these spaces are closely related to a certain moduli space
${\mathcal C\/}_{k+1,n}$ of $n$-gons
in $k$-dimensional projective space.
While ${\mathcal C\/}_{k+1,n}$ can be viewed as the quotient of the
Grassmannian $\Gr_{k+1,n}$ by a torus action, see~\cite{GM},
the space  ${\mathcal E\/}_{k+1,n}$ is a subvariety of~$\Gr_{k+1,n}$.
We show that, in the case where $k+1$ and $n$ (and thus $w+1$ and $n$) are coprime, 
the two spaces are isomorphic.

A frieze pattern is just another, more combinatorial, way to represent
a linear difference equation with (anti)periodic solutions.
However, the theory of frieze patterns was created~\cite{Cox,CoCo}
and developed independently, see~\cite{Pro,ARS,BR,MGOT,MG}.
The current interest to this subject is due to the relation with the theory
of cluster algebras~\cite{CaCh}.
We show that the isomorphism between  ${\mathcal E\/}_{k+1,n}$ and
${\mathcal F\/}_{k+1,n}$
immediately implies the periodicity statement which is an important part of the theory.
Let us  mention that a more general notion of $\SL_{k+1}$-tiling was studied in~\cite{BR};
a version of $\SL_{3}$-frieze patterns, called $2$-frieze patterns, was studied 
in~\cite{Pro,MGOT}.

Let us also mention that an $\SL_{k+1}$-frieze pattern can be included in a larger pattern
sometimes called a $T$-system (see~\cite{DK} and references therein), 
and better known under the name of
{\it discrete Hirota equation} (see,  for a survey,~\cite{Zab}).
Although an $\SL_{k+1}$-frieze pattern is a small part of a solution of a $T$-system,
it contains all the information about the solution.
We will not use this viewpoint in the present paper.

The main result of this paper is a description of the duality
between the spaces ${\mathcal E\/}_{k+1,n}$ and ${\mathcal E\/}_{w+1,n}$
when $(k+1)+(w+1)=n$.
We call this duality the {\it combinatorial Gale transform}.
This is an analog of the classical Gale transform
which is a well-known duality on the moduli spaces of point configurations,
see \cite{Gal,Cob,Cob1,EP}.
We think that the most interesting feature of the combinatorial Gale transform
is that it allows one to change the order of an equation keeping
all the information about it.

Let us give here the simplest example, 
which is related to Gauss' {\it pentagramma mirificum}~\cite{Gau}.
Consider a third-order difference equation
$$
V_{i}=a_{i}V_{i-1}-b_iV_{i-2}+V_{i-3},
\qquad i\in\Z.
$$
Assume that the coefficients $(a_i)$ and $(b_i)$
are $5$-periodic: $a_{i+5}=a_{i}$ and $b_{i+5}=b_{i}$,
and that all the solutions $(V_i)$ are
also $5$-periodic, i.e., $V_{i+5}=V_i$.
The combinatorial Gale transform in this case consists
in ``forgetting the coefficients $b_i$''; to this equation it associates 
the difference equation of order~$2$:
$$
W_{i}=a_{i}W_{i-1}-W_{i-2}.
$$
It turns out that the solutions of the latter equation are $5$-antiperiodic:
$W_{i+5}=-W_{i}$.
Conversely, as we will explain later,
one can reconstruct the initial third order equation
from the second order one.
Geometrically speaking, this transform sends
(projective equivalence classes of) pentagons in~$\PP^2$ to those in~$\PP^1$. 
In terms of the frieze patterns, this corresponds to a duality between
$5$-periodic Coxeter friezes and $5$-periodic $\SL_3$-friezes.

Our study is motivated by recent works
\cite{OST,Sol,OST1,GSTV,Sch2,MB,MB1,KS,KS1}
on a certain class of discrete integrable systems arising in 
projective differential geometry and cluster algebra.
The best known among these maps is the pentagram map~\cite{Sch,Sch1}
acting on the moduli space of $n$-gons in the projective plane.

This paper is organized as follows.

In Section~\ref{OblectS}, we introduce the main objects,
namely the spaces
${\mathcal E\/}_{k+1,n}$,
${\mathcal F\/}_{k+1,n}$, and
${\mathcal C\/}_{k+1,n}$.

In Section~\ref{GeoS}, 
we construct an embedding of ${\mathcal E\/}_{k+1,n}$
into the Grassmannian $\Gr_{k+1,n}$.
We then formulate the result about the isomorphism between 
${\mathcal E\/}_{k+1,n}$ and ${\mathcal F\/}_{k+1,n}$ and
give an explicit construction of this isomorphism.
We also define a natural map from the space of equations to
that of configurations of points in the projective space.

In Section~\ref{TheGaleS}, we introduce the
Gale transform which is the main notion of this paper.
We show that a difference equation corresponds not to just one,
but to two different frieze patterns.
This is what we call the Gale duality.
We also introduce a more elementary notion of projective duality
that commutes with the Gale transform.

In Section~\ref{DeTS}, we calculate explicitly the entries
of the frieze pattern associated with a difference equation.
We give explicit formulas for the Gale transform.
These formulas are similar to the classical and well-known expressions
often called the Andr\'e determinants, see~\cite{And,Jor}.

Relation of the Gale transform to representation theory is
described in Section~\ref{RePSS}.
We represent an $\SL_{k+1}$-frieze pattern 
(and thus a difference equation) in a form of a
unitriangular matrix.
We prove that the Gale transform coincides with the restriction of the
involution of the nilpotent group of unitriangular matrices
introduced and studied by Berenstein, Fomin, and Zelevinsky~\cite{BFZ}.

In Section~\ref{PerSec}, we present an application of
the isomorphism between difference equations and frieze patterns.
We construct rational
periodic maps generalizing the Gauss map. 
These maps are obtained by calculating consecutive coefficients
of (anti)periodic second and third order difference equations and using
the periodicity property of
$\SL_{2}$- and $\SL_{3}$-frieze patterns.
We also explain how these rational maps can be derived, in an alternate way,
from the projective geometry of polygons.

In Section~\ref{TriThmS}, 
we explain the details about the relations between 
the spaces we consider.
We also outline  relations to the Teichm\"uller theory.

\section{Difference equations, $\SL_{k+1}$-frieze patterns and polygons in $\RP^k$}\label{OblectS}

In this chapter we will define the three closely related spaces
${\mathcal E\/}_{k+1,n}$,
${\mathcal F\/}_{k+1,n}$, and
${\mathcal C\/}_{k+1,n}$ discussed in the introduction.
All three spaces will be equipped with the structure of algebraic variety;
we choose $\R$ as the ground field.

\subsection{Difference equations}
Let ${\mathcal E}_{k+1,n}$ be the space of order $k+1$ difference equations 
\begin{equation}
\label{REq}
V_{i}=a_{i}^1V_{i-1}-a_{i}^{2}V_{i-2}+ \cdots+(-1)^{k-1}a_{i}^{k}V_{i-k}+(-1)^{k}V_{i-k-1},
\end{equation}
where $a_{i}^{j}\in \R$, with $i\in\Z$ and $1\leq j\leq k$, are coefficients 
and $V_i$ are unknowns.
(Note that the superscript $j$ is an index, not a power.)
Throughout the paper, a solution $(V_i)$ will consist of either real numbers, $V_i\in \R$,
or of real vectors $V_i\in\R^{k+1}$.
We always assume that the coefficients are {\it periodic} with some period $n\geq{}k+2$:
$$
a_{i+n}^{j}=a_{i}^{j}
$$ 
for all $i,j$, and that all the solutions are {\it $n$-(anti)periodic}: 
\begin{equation}
\label{APeriod}
V_{i+n}=(-1)^{k}\,V_{i}.
\end{equation}

The algebraic variety structure on ${\mathcal E}_{k+1,n}$ will be introduced
in Section~\ref{AlgVarS}.

\begin{example}
The simplest example of a difference equation (\ref{REq}) is the 
well-known discrete {\it Hill}
(or {\it Sturm-Liouville}) equation
$$
V_{i}=a_{i}V_{i-1}-V_{i-2}
$$
with $n$-periodic coefficients and $n$-antiperiodic solutions.
\end{example}

\subsection{$\SL_{k+1}$-frieze patterns}
An $\SL_{k+1}$-{\it frieze pattern} (see~\cite{BR}) is an infinite array of numbers such that every 
$(k+1)\times (k+1)$-subarray with adjacent rows and columns forms an element of $\SL_{k+1}$.

More precisely,
the entries of the frieze are organized in an infinite strip. 
The entries are denoted by $(d_{i,j})$,
where $i\in\Z$, and 
$$
i-k-1\leq{}j\leq{}i+w+k,
$$ 
 for a fixed $i$, and satisfy the following ``boundary conditions''
$$
\left\{
\begin{array}{rccccl}
d_{i,i-1}&=&d_{i,i+w}&=&1& \hbox{for all}\; i,\\[4pt]
d_{i,j}&=&0&&&\hbox{for}\;j<i-1\;\hbox{or}\;j>i+w,
\end{array}
\right.
$$
and the ``$\SL_{k+1}$-conditions''
\begin{equation}
\label{DEq}
D_{i,j}:=
\left\vert
\begin{array}{llll}
d_{i,j}&d_{i,j+1}&\ldots&d_{i,j+k}\\[4pt]
d_{i+1,j}&d_{i+1,j+1}&\ldots&d_{i+1,j+k}\\[4pt]
\ldots& \ldots&& \ldots\\
d_{i+k,j}&d_{i+k,j+1}&\ldots&d_{i+k,j+k}
\end{array}
\right\vert=1,
\end{equation}
for all $(i,j)$ in the index set.

An $\SL_{k+1}$-frieze pattern is represented as follows
\begin{equation}
\label{FREq}
\begin{array}{ccccccccccccc}
&&&& \vdots&&&& \vdots&&&\\
&0&&0&&0&&0&&0&&\ldots\\[6pt]
\ldots&&1&&1&&1&&1&&1&\\[6pt]
&\ldots&&\;d_{0,w-1}&&\;d_{1,w}&&\;d_{2,w+1}&&\ldots&&\ldots\\
&&&\! \iddots&& \iddots&& \iddots&&&&\\
\ldots&& d_{0,1}&&d_{1,2}&&d_{2,3}&&d_{3,4}&&d_{4,5}&\\[6pt]
& d_{0,0}&&d_{1,1}&&d_{2,2}&&d_{3,3}&&d_{4,4}&&\ldots\\[6pt]
\ldots&&1&&1&&1&&1&&1&\\[6pt]
&0&&0&&0&&0&&0&&\ldots\\
&&&& \vdots&&&& \vdots&&&\\
\end{array}
\end{equation}
where the strip is bounded by $k$ rows of 0's at the top, and at the bottom.
To simplify the pictures we often omit the bounding rows of 0's.

The parameter $w$ is called the {\it width} of the $\SL_{k+1}$-frieze pattern.
In other words, the width
is the number of non-trivial rows between  the rows of~$1$'s. 

\begin{definition}
\label{tamedef}
An $\SL_{k+1}$-frieze pattern is called {\it tame}
if every $(k+2)\times (k+2)$-determinant equals $0$.
\end{definition}

The notion of tame friezes was introduced in \cite{BR}.
Let ${\mathcal F}_{k+1,n}$ denote the space of tame $\SL_{k+1}$-frieze patterns 
of width $w=n-k-2$.

\begin{example}

(a)
The most classical example of friezes is that of 
{\it Coxeter-Conway frieze patterns} \cite{Cox,CoCo} corresponding to $k=1$.
For instance, a generic Coxeter-Conway frieze pattern of width~$2$ looks like this:
$$
 \begin{array}{ccccccccccc}
\cdots&&1&& 1&&1&&\cdots
 \\[4pt]
&x_1&&\frac{x_2+1}{x_1}&&\frac{x_1+1}{x_2}&&x_2&&
 \\[4pt]
\cdots&&x_2&&\frac{x_1+x_2+1}{x_1x_2}&&x_1&&\cdots
 \\[4pt]
&1&&1&&1&&1&&
\end{array}
$$
for some $x_1,x_2$.
(Note that we omitted the first and the last rows of $0$'s.)
This example is related to so-called Gauss {\it pentagramma mirificum}~\cite{Gau},
see also~\cite{Sche}.

(b)
The following  width $3$ Coxeter pattern is not tame:
$$
 \begin{array}{cccccccccccc}
\cdots&&1&& 1&&1&&1&&\cdots
 \\[4pt]
&1&&0&&2&&0&&3&
 \\[4pt]
\cdots&&-1&&-1&&-1&&-1&&\cdots
 \\[4pt]
&0&&1&&0&&2&&0&
 \\[4pt]
\cdots&&1&&1&&1&&1&&\cdots
\end{array}
$$
\end{example}

 \begin{remark}
Generic $\SL_{k+1}$-frieze patterns are tame.
We understand the genericity of an $\SL_{k+1}$-frieze pattern as the condition that
every $k\times k$-determinant is different from $0$.
Then the vanishing of the $(k+2)\times (k+2)$-determinants follows from
the Dodgson condensation formula, involving minors of order $k+2$, $k+1$ an $k$ obtained by erasing the first and/or last row/column.
The formula can be pictured as follows
$$
\begin{vmatrix}
*&*&*&*\\
*&*&*&*\\
*&*&*&*\\
*&*&*&*\\
\end{vmatrix}
\begin{vmatrix}
&&&\\
\;&*&*&\,\\
\;&*&*&\,\\
\;&&&\,\\
\end{vmatrix}
=
\begin{vmatrix}
\;*&*&*&\;\; \\
\;*&*&*&\;\;\\
\;*&*&*&\;\;\\
&&&\;\;\\
\end{vmatrix}
\begin{vmatrix}
\;&\;&&\\
\;&\;*&*&*\;\;\\
\;&\;*&*&*\;\;\\
\;&\;*&*&*\;\;\\
\end{vmatrix}
-
\begin{vmatrix}
&\;*&*&*\;\;\\
&\;*&*&*\;\;\\
&\;*&*&*\;\;\\
&\;&&\\
\end{vmatrix}
\begin{vmatrix}
&&&\;\\
\;*&*&*&\;\;\\
\;*&*&*&\;\;\\
\;*&*&*&\;\;\\
\end{vmatrix}.
$$
where the deleted columns/rows are left blank.
\end{remark}

\begin{notation}
\label{NotNot}
Given an $\SL_{k+1}$-frieze pattern $F$ as in \eqref{FREq},
it will be useful to define the following  $(k+1)\times{}n$-matrices\footnote{
Throughout this paper, the ``empty'' entries of matrices stand for $0$.}
\begin{equation}
\label{TheEmbM}
M^{(i)}_F:=\left(
\begin{array}{cccccccccccccc}
 1 &   d_{i,i}& \ldots   &\ldots&d_{i,w+i-1}&1 \\[4pt]
& \ddots & &  &&&\ddots\\[12pt]
&    & 1 & d_{k+i,k+i} &\ldots& \ldots&d_{k+i,k+w+i-1} &1\\[4pt]
\end{array}
\right).
\end{equation}
For a subset $I$ of $k+1$ consecutive elements of $\Z/n\Z$,
denote by $\Delta_I(M)$ the minor of a matrix $M$ based on columns with indices in $I$.
There are $n$ such intervals $I$. 
By definition of $\SL_{k+1}$-frieze pattern, $\Delta_I(M^{(i)}_F)=1$.

The matrix $M^{(i)}_F$ determines a unique tame $\SL_{k+1}$-frieze pattern.
Indeed, from $M^{(i)}_F$ one can compute all the entries $d_{i,j}$,
one after another, using the fact that $(k+2)\times (k+2)$-minors
of the frieze pattern vanish.

We will also denote the
North-East (resp. South-East) diagonals of an $\SL_{k+1}$-frieze pattern by $\mu_i$ (resp. $\eta_j$):

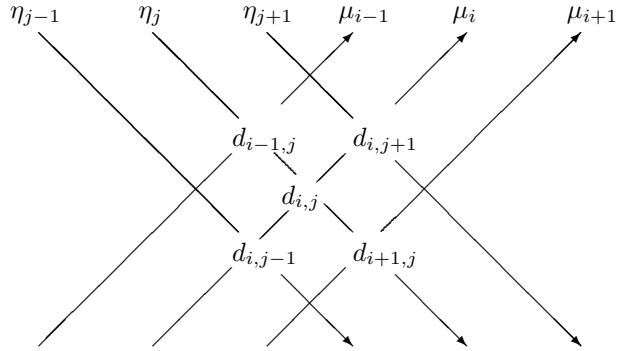
\begin{center}
\setlength{\unitlength}{3144sp}%
\begingroup\makeatletter\ifx\SetFigFont\undefined%
\gdef\SetFigFont#1#2#3#4#5{%
  \reset@font\fontsize{#1}{#2pt}%
  \fontfamily{#3}\fontseries{#4}\fontshape{#5}%
  \selectfont}%
\fi\endgroup%
\begin{picture}(4813,2916)(1576,-3898)
\put(3713,-2761){\makebox(0,0)[lb]{\smash{{\SetFigFont{10}{16.8}{\rmdefault}{\mddefault}{\updefault}{\color[rgb]{0,0,0}$d_{i,j}$}%
}}}}
\thinlines
{\color[rgb]{0,0,0}\multiput(3657,-2367)(5.63333,-5.63333){31}{\makebox(1.5875,11.1125){\SetFigFont{5}{6}{\rmdefault}{\mddefault}{\updefault}.}}
}%
{\color[rgb]{0,0,0}\put(3319,-2367){\line(-1,-1){1518.500}}
}%
{\color[rgb]{0,0,0}\put(3319,-3267){\line(-1,-1){618.500}}
}%
{\color[rgb]{0,0,0}\put(4219,-3267){\line(-1,-1){618.500}}
}%
{\color[rgb]{0,0,0}\multiput(4613,-2873)(-5.62500,-5.62500){21}{\makebox(1.5875,11.1125){\SetFigFont{5}{6}{\rmdefault}{\mddefault}{\updefault}.}}
}%
{\color[rgb]{0,0,0}\put(3769,-2817){\line(-1,-1){225}}
}%
{\color[rgb]{0,0,0}\put(4219,-2367){\line(-1,-1){225}}
}%
{\color[rgb]{0,0,0}\put(3713,-3323){\vector( 1,-1){563}}
}%
{\color[rgb]{0,0,0}\put(4613,-3323){\vector( 1,-1){563}}
}%
{\color[rgb]{0,0,0}\put(4613,-2423){\vector( 1,-1){1463}}
}%
{\color[rgb]{0,0,0}\put(4613,-1973){\vector( 1, 1){562.500}}
}%
{\color[rgb]{0,0,0}\put(4613,-2873){\vector( 1, 1){1462.500}}
}%
{\color[rgb]{0,0,0}\put(3713,-1973){\vector( 1, 1){562.500}}
}%
{\color[rgb]{0,0,0}\put(4276,-2086){\line(-1, 1){675}}
\put(3601,-1411){\line( 1,-1){675}}
}%
{\color[rgb]{0,0,0}\put(3376,-2086){\line(-1, 1){675}}
\put(2701,-1411){\line( 1,-1){675}}
}%
{\color[rgb]{0,0,0}\put(3376,-2986){\line(-1, 1){1575}}
\put(1801,-1411){\line( 1,-1){1575}}
}%
\put(1576,-1298){\makebox(0,0)[lb]{\smash{{\SetFigFont{10}{16.8}{\rmdefault}{\mddefault}{\updefault}{\color[rgb]{0,0,0}$\eta_{j-1}$}%
}}}}
\put(5963,-1298){\makebox(0,0)[lb]{\smash{{\SetFigFont{10}{16.8}{\rmdefault}{\mddefault}{\updefault}{\color[rgb]{0,0,0}$\mu_{i+1}$}%
}}}}
\put(5063,-1298){\makebox(0,0)[lb]{\smash{{\SetFigFont{10}{16.8}{\rmdefault}{\mddefault}{\updefault}{\color[rgb]{0,0,0}$\mu_{i}$}%
}}}}
\put(4163,-1298){\makebox(0,0)[lb]{\smash{{\SetFigFont{10}{16.8}{\rmdefault}{\mddefault}{\updefault}{\color[rgb]{0,0,0}$\mu_{i-1}$}%
}}}}
\put(3408,-1298){\makebox(0,0)[lb]{\smash{{\SetFigFont{10}{16.8}{\rmdefault}{\mddefault}{\updefault}{\color[rgb]{0,0,0}$\eta_{j+1}$}%
}}}}
\put(2588,-1298){\makebox(0,0)[lb]{\smash{{\SetFigFont{10}{16.8}{\rmdefault}{\mddefault}{\updefault}{\color[rgb]{0,0,0}$\eta_{j}$}%
}}}}
\put(4276,-3211){\makebox(0,0)[lb]{\smash{{\SetFigFont{10}{16.8}{\rmdefault}{\mddefault}{\updefault}{\color[rgb]{0,0,0}$d_{i+1,j}$}%
}}}}
\put(3326,-2311){\makebox(0,0)[lb]{\smash{{\SetFigFont{10}{16.8}{\rmdefault}{\mddefault}{\updefault}{\color[rgb]{0,0,0}$d_{i-1,j}$}%
}}}}
\put(4276,-2311){\makebox(0,0)[lb]{\smash{{\SetFigFont{10}{16.8}{\rmdefault}{\mddefault}{\updefault}{\color[rgb]{0,0,0}$d_{i,j+1}$}%
}}}}
\put(3326,-3211){\makebox(0,0)[lb]{\smash{{\SetFigFont{10}{16.8}{\rmdefault}{\mddefault}{\updefault}{\color[rgb]{0,0,0}$d_{i,j-1}$}%
}}}}
{\color[rgb]{0,0,0}\put(4051,-2761){\line( 1,-1){225}}
}%
\end{picture}%
\end{center}
\end{notation}

\subsection{Moduli space of polygons}
A {\it non-degenerate} $n$-gon is a map 
$$
v:\Z\to\RP^{k}
$$ 
such that $v_{i+n}=v_i$, for all $i$, and no $k+1$ consecutive vertices belong to
the same hyperplane.

Let ${\mathcal C}_{k+1,n}$ be the moduli space of projective equivalence classes of 
non-degenerate $n$-gons in~$\RP^{k}$.

\begin{remark}
The space ${\mathcal C}_{k+1,n}$ has been extensively studied; see, e.g., \cite{Gal,GM,Kap,EP}. 
Our interest in this space is motivated by the study of the pentagram map,
a completely integrable map on the space ${\mathcal C}_{3,n}$;
see \cite{Sch,Sch1,OST,Sol,OST1}.
\end{remark}

\section{Geometric description of the spaces ${\mathcal E}_{k+1,n}$,
${\mathcal F}_{k+1,n}$ and ${\mathcal C}_{k+1,n}$}\label{GeoS}

In this section, we describe the structures of algebraic varieties
on the spaces ${\mathcal E}_{k+1,n}$,
${\mathcal F}_{k+1,n}$ and ${\mathcal C}_{k+1,n}$.
We also prove that the spaces ${\mathcal E}_{k+1,n}$ and
${\mathcal F}_{k+1,n}$ are isomorphic.
The spaces ${\mathcal E}_{k+1,n}$ and ${\mathcal C}_{k+1,n}$ are also closely related.
We will define a map from ${\mathcal E}_{k+1,n}$ to ${\mathcal C}_{k+1,n},$
which turns out to be an isomorphism, provided $k+1$ and $n$ are coprime.
If this is not the case, then these two spaces are different.

\subsection{The structure of an algebraic variety on ${\mathcal E}_{k+1,n}$}\label{AlgVarS}

The space of difference equations~${\mathcal E}_{k+1,n}$
is an affine algebraic subvariety of $\R^{nk}$.
This structure is defined by the condition~(\ref{APeriod}).

The space of solutions of the equation~\eqref{REq} is $(k+1)$-dimensional. 
Consider $k+1$ linearly independent solutions forming 
a sequence $(V_i)_{i\in \Z}$ of vectors in $\R^{k+1}$ satisfying \eqref{REq}.
Since the coefficients are $n$-periodic,
there exists a linear map $T$ on the space of solutions 
called the {\it monodromy} satisfying:
$$
V_{i+n}=TV_i,
$$
One can view the monodromy as an element of $\GL_{k+1}$
defined up to a conjugation.

\begin{proposition}
\label{monoid}
The space ${\mathcal E}_{k+1,n}$
has codimension $k(k+2)$ in $\R^{nk}$ with coordinates $a_{i}^{j}$. 
\end{proposition}

\begin{proof} 
Since the last coefficient in (\ref{REq}) is $(-1)^{k}$, one has\footnote{
Throughout the paper 
$\left|V_i,\ldots,V_{i+k}\right|$ stand for the determinant 
of the matrix with columns $V_i,\ldots,V_{i+k}$.}
$$
\left|V_i,V_{i+1},\ldots,V_{i+k}\right| = 
\left|V_{i+1},V_{i+2},\ldots,V_{i+k+1}\right|.
$$
Hence the monodromy is volume-preserving and thus belongs to
the group~$\SL_{k+1}(\R)$. 

If furthermore all the solutions of (\ref{REq})
are $n$-(anti)periodic, then the monodromy is $(-1)^{k} \mathrm{Id}$.
Since $\dim \SL_{k+1}=k(k+2)$,
this gives $k(k+2)$ polynomial relations on the coefficients. 
\end{proof}

\subsection{Embedding of ${\mathcal F}_{k+1,n}$ into the Grassmannian}\label{SubSGr}

Consider the  Grassmannian  $\Gr_{k+1,n}$ of $(k+1)$-dimensional
subspaces in~$\R^n$.
Let us show that
the space of $\SL_{k+1}$-frieze patterns~${\mathcal F}_{k+1,n}$ can be viewed
as an $(n-1)$-codimensional algebraic subvariety 
of  $\Gr_{k+1,n}$.

Recall that the the  Grassmannian can be described as the quotient
\begin{equation}
\label{Grass}
\Gr_{k+1,n}\simeq
\GL_{k+1}\backslash \Mat^*_{k+1,n}(\R),
\end{equation}
where $\Mat^*_{k+1,n}(\R)$ is the set of real $(k+1)\times{}n$-matrices
of rank $k+1$.
Every point of $\Gr_{k+1,n}$ can be represented by a $(k+1)\times{}n$-matrix $M$.
The Pl\"ucker coordinates on $\Gr_{k+1,n}$
are all the $(k+1)\times{}(k+1)$-minors of $M$, see, e.g., \cite{Ful}. 

Note that the $(k+1)\times{}(k+1)$-minors depend on the
choice of the matrix but they are defined up to a common factor.
Therefore, the Pl\"ucker coordinates are homogeneous coordinates 
independent of the choice of the
representing matrix. 
The minors $\Delta_I(M)$ with consecutive columns, see Notation~\ref{NotNot},
play a special role.

\begin{proposition}
\label{EmbProp}
The subvariety of 
$Gr_{k+1,n}$ consisting in the elements that can be represented by
the matrices such that all the minors $\Delta_I(M)$ are equal to each other:
\begin{equation}
\label{EEqu}
\Delta_I(M)=\Delta_{I'}(M)
\qquad\hbox{for all}\quad
I,I'
\end{equation}
is in one-to-one correspondence with the space $\Fc_{k+1,n}$
of $\SL_{k+1}$-frieze patterns.
\end{proposition}

\begin{proof}
If $M\in\Gr_{k+1,n}$ satisfies $\Delta_I(M)=\Delta_{I'}(M)$ for all intervals $I,I'$, 
then it is easy to see that there is a unique representative of $M$ 
of the form (\ref{TheEmbM}), 
and therefore a unique corresponding $\SL_{k+1}$-frieze pattern.

Conversely, given an $\SL_{k+1}$-frieze pattern $F$, the corresponding
matrices $M_F^{(i)}$ satisfy~\eqref{EEqu}.
Fixing $i$, for instance, taking $i=1$, we obtain a well-defined embedding
$$
{\mathcal F}_{k+1,n}\subset\Gr_{k+1,n}.
$$
Hence the result.
\end{proof}

Note that the constructed embedding of the space 
${\mathcal F}_{k+1,n}$ into the Grassmannian
depends on the choice of index $i$.
A different choice leads to a different embedding.

\subsection{The space ${\mathcal C}_{k+1,n}$ as a quotient of the Grassmannian}\label{QuotGr}

A classical way to describe the space ${\mathcal C}_{k+1,n}$ 
 as the quotient of $\Gr_{k+1,n}$ by the torus action:
\begin{equation}
\label{TheQuot}
{\mathcal C}_{k+1,n}\simeq\Gr_{k+1,n}\slash\T^{n-1},
\end{equation}
is  due to Gelfand and MacPherson~\cite{GM}.

Let us comment on this realization of ${\mathcal C}_{k+1,n}$.
Given an $n$-gon $v:\Z\to\RP^k$,  consider an arbitrary lift
$V:\Z\to\R^{k+1}$.
The result of such a lift is a full rank $(k+1)\times{}n$-matrix,
and thus an element of $\Gr_{k+1,n}$.
Recall that the action of $\T^{n-1}$, in terms of the matrix realization (\ref{Grass}),
consists in multiplying $(k+1)\times{}n$-matrices by diagonal
$n\times{}n$-matrices with determinant~$1$.
The projection of $V$ to the quotient $\Gr_{k+1,n}\slash\T^{n-1}$
is independent of the lift of $v$ to $\R^{k+1}$.

\subsection{Triality}\label{1stIso}
Let us briefly explain the relations between the spaces
${\mathcal  E}_{k+1,n}$, ${\mathcal  F}_{k+1,n}$ and ${\mathcal  C}_{k+1,n}$.
We will give more details in Section~\ref{TriThmS}.

The spaces of difference equations ${\mathcal  E}_{k+1,n}$ and that of $\SL_{k+1}$-frieze patterns
${\mathcal  F}_{k+1,n}$ are always isomorphic.
These spaces are, in general, different from the moduli space of $n$-gons, but
isomorphic to it if $k+1$ and $n$ are coprime.

\begin{theorem}
\label{TriThm}
(i)
The spaces
${\mathcal  E}_{k+1,n}$ and ${\mathcal  F}_{k+1,n}$ are isomorphic
algebraic varieties.

(ii) If $k+1$ and $n$ are coprime, then 
the spaces ${\mathcal  E}_{k+1,n}$, ${\mathcal  F}_{k+1,n}$ and ${\mathcal  C}_{k+1,n}$ 
are isomorphic algebraic varieties. 
\end{theorem}

The complete proof of this theorem will be given in Section~\ref{TriThmS}.
Here we just construct the isomorphisms.

\medskip

{\it Part (i)}.
Let us define a map 
$$
{\mathcal  E}_{k+1,n}
\buildrel{\simeq}\over\longrightarrow 
{\mathcal  F}_{k+1,n}
$$
which identifies the two spaces.
Roughly speaking, we generate the solutions of the recurrence 
equation (\ref{REq}), starting with $k$ zeros, 
followed by one, and put them on the North-East diagonals of the frieze pattern.
Let us give a more detailed construction.

Given a difference equation \eqref{REq}
satisfying the (anti)periodicity assumption \eqref{APeriod},
we define the corresponding $\SL_{k+1}$-frieze pattern
(\ref{FREq}) by constructing its North-East diagonals $\mu_i$.
This diagonal is given by a sequence of real numbers 
$V=(V_s)_{s\in \Z}$ that are the solution of the equation \eqref{REq} 
with the initial condition
$$
(V_{i-k-1},V_{i-k},\ldots,V_{i-1})=(0,0,\ldots,0,1);
$$ 
this
defines the numbers $d_{i,j}$ via
\begin{equation}
\label{ConstEq}
d_{i,j}:=V_j.
\end{equation}
Since the solution is $n$-(anti)periodic, we have
$$
d_{i,i+w+1}=\ldots=d_{i,i+n-2}=0,
\qquad
d_{i,i+n-1}=(-1)^k.
$$
Furthermore, from the equation (\ref{REq}) one has
$$
d_{i,i+w}=1.
$$
The sequence of (infinite) vectors $(\eta_i)$, i.e., of South-East diagonals,
satisfies Equation~(\ref{REq}).

\begin{example}
\label{HillEx}
{\rm
For an arbitrary difference equation,
the first coefficients of the $i$-th diagonal are
$$
\begin{array}{rcl}
d_{i,i}&=&a_{i}^1, \\[4pt] 
d_{i,i+1}&=&a_{i}^1a_{i+1}^1-a_{i+1}^2, \\[4pt]
d_{i,i+2}&=&a_{i}^1a_{i+1}^1a_{i+2}^1-a_{i+1}^2a_{i+2}^1- a_{i}^1a_{i+2}^2+a_{i+2}^3.
\end{array}
$$
We will give more general determinant formulas in Section~\ref{DeTS}.
}
\end{example}

\begin{remark}
The idea that the diagonals of a frieze pattern satisfy a difference equation
goes back to Conway and Coxeter~\cite{CoCo}.
The isomorphism between the spaces of difference
equations and frieze patterns was used 
in \cite{OST,MGOT} in the cases $k=1,2$.
\end{remark}

{\it Parts (ii)}.
The construction of the  map 
${\mathcal E}_{k+1,n}\to{\mathcal C}_{k+1,n}$ consists in the following two steps:
\begin{enumerate}
\item
the spaces ${\mathcal E}_{k+1,n}$ and ${\mathcal F}_{k+1,n}$ are isomorphic;
\item
there is an embedding ${\mathcal F}_{k+1,n}\subset\Gr_{k+1,n}$,  
and a projection
$\Gr_{k+1,n}\to{\mathcal C}_{k+1,n}$, see~(\ref{TheEmbM}) and~(\ref{TheQuot}).
\end{enumerate}

More directly, given a difference equation (\ref{REq}) with (anti)periodic solutions,
the space of solutions being $k+1$-dimensional, we
choose any linearly independent solutions
$(V_i^{(1)}),\ldots,(V_i^{(k+1)})$.
For every $i$, we obtain a point, $V_i\in \R^{k+1}$, which we project to $\RP^k$;
the (anti)periodicity assumption implies that we obtain an $n$-gon.
Furthermore, the constructed $n$-gon is non-degenerate since the 
$(k+1)\times (k+1)$-determinant
\begin{equation}
\label{detconst}
\left|
V_i, V_{i+1},\ldots,V_{i+k}
\right|=\mathrm{Const}\not=0.
\end{equation}
A different choice of  solutions leads to a projectively equivalent $n$-gon.
We have constructed a map
\begin{equation}
\label{MaPP}
{\mathcal E}_{k+1,n} \longrightarrow {\mathcal C}_{k+1,n}.
\end{equation}

We will see in Section~\ref{CoPSeC}, 
that this map is an isomorphism if and only if $k+1$ and $n$ are coprime.

\begin{proposition}
\label{NewProp}
Suppose that $\gcd(n,k+1)=q \neq 1$,
then the image of the constructed map has codimension $q-1$.
\end{proposition}

We will give the proof in Section~\ref{ProProSec}.

\begin{notation}
It will be convenient to write the $\SL_{k+1}$-frieze pattern
associated with a difference equation (\ref{REq}) in the form:
\begin{equation}
\label{friezeGG}
\begin{array}{ccccccccccccccccccccccccccc}
&\ldots&1&&1&&1&&1&&1&&1\\[4pt]
&&&\ldots&&\alpha^{1}_{n}&&\alpha^{1}_1&&\alpha^{1}_2&& \ldots&&\alpha^{1}_n&\\[4pt]
&&&&\alpha^{2}_n&&\alpha^{2}_1&& \alpha^{2}_2&&&&\alpha^{2}_n&\\
&\ldots&& \iddots && \iddots&& \iddots&&&& \iddots&&\ldots\\
&&\alpha^{w}_{n}&& \alpha^{w}_{1}&&\alpha^{w}_{2}&& \ldots&& \alpha^{w}_{n}&&\ldots\\[4pt]
&1&&1&&1&&1&&1&&1&&\ldots\\
\end{array}
\end{equation}
The relation between the old and new notation for the entries of the $\SL_{k+1}$-frieze pattern 
of width $w$ is:
\begin{equation}
\label{Corresp}
d_{i,j}=\alpha_{i-1}^{w-j+i},\qquad
\alpha_{i}^{j}=d_{i+1,w+i-j+1},
\end{equation}
where $d_{i,j}$ is the general notation 
for the entries of an $\SL_{k+1}$-frieze pattern.
See formula~(\ref{FREq}).
\end{notation}

\section{The combinatorial Gale transform}\label{TheGaleS}

In this section, we present another isomorphism between
the introduced spaces. 
This is a combinatorial analog of the classical Gale transform, and
it  results from the natural isomorphism of Grassmannians:
$$
\Gr_{k+1,n}\simeq\Gr_{w+1,n},
$$
for $n=k+w+2$.

On the spaces $\Ec_{k+1,n}$ and $\Fc_{k+1,n}$ we also define an involution,
which is a combinatorial version of the projective duality.
The Gale transform commutes with the projective duality
so that both maps define an action of the Klein group 
$\left(\Z/2\Z\right)^2$.

\subsection{Statement of the result}\label{StatS}

\begin{definition}
We say that a difference equation~(\ref{REq}) with $n$-(anti)periodic solutions 
is {\it Gale dual} of
the following difference equation of order $w+1$:
\begin{equation}
\label{TheDualEq}
W_i=\alpha_{i}^1W_{i-1}-\alpha_{i}^2W_{i-2}+ 
\cdots+(-1)^{w-1}\alpha_{i}^wW_{i-w}+(-1)^{w}W_{i-w-1},
\end{equation}
where $\alpha_r^s$ are the entries of the $\SL_{k+1}$-frieze pattern~(\ref{friezeGG})
corresponding to~(\ref{REq}).
\end{definition}

The following statement is the main result of the paper.

\begin{theorem}
\label{TheGailThm}
(i)
All  solutions of the equation~(\ref{TheDualEq}) are $n$-(anti)periodic,
i.e., 
$$
W_{i+n}=(-1)^wW_i.
$$

(ii)
The defined map  ${\mathcal E}_{k+1,n}\to{\mathcal E}_{w+1,n}$ is an involution.
\end{theorem}

We obtain an isomorphism
\begin{equation}
\label{TheGale}
\Gc:{\mathcal E}_{k+1,n} 
\buildrel{\simeq}\over\longrightarrow{\mathcal E}_{w+1,n}
\end{equation}
between the spaces of $n$-(anti)periodic difference equations of orders $k+1$ 
and $w+1$, provided 
$$
n=k+w+2.
$$ 
We call this isomorphism the \textit{combinatorial Gale duality} 
(or the \textit{combinatorial Gale transform}).

An equivalent way to formulate the above theorem
is to say that there is a duality between 
$\SL_{k+1}$-frieze patterns of width $w$
and $\SL_{w+1}$-frieze patterns of width $k$:
$$
\Gc:{\mathcal F}_{k+1,n} 
\buildrel{\simeq}\over\longrightarrow{\mathcal F}_{w+1,n}.
$$

\begin{proposition}
\label{GDFprop}
The $\SL_{w+1}$-frieze pattern associated
to the equation~(\ref{TheDualEq}) is the following
\begin{equation}
\label{friezeG}
\begin{array}{ccccccccccccccccccccccccccc}
&\ldots&1&&1&&1&&1&&1&&1\\[4pt]
&&&\ldots&& a^1_n&&a^1_1&&a^1_2&& \ldots&& a^1_n&\\[4pt]
&&&& a^2_n&&a^2_1&& a^2_2&&&&a^2_n&\\
&\ldots&& \iddots && \iddots&& \iddots&&&& \iddots&&\ldots\\
&&a^k_n&& a^k_1&&a^k_2&& \ldots&& a^k_n&&\ldots\\[4pt]
&1&&1&&1&&1&&1&&1&&\ldots\\
\end{array}
\end{equation}
\end{proposition}

We say that the $\SL_{w+1}$-frieze pattern~(\ref{friezeG})
is {\it Gale dual} to the $\SL_{k+1}$-frieze pattern~(\ref{friezeGG}).

The combinatorial Gale transform is illustrated by Figure~\ref{Illustre}.

 \begin{center}
 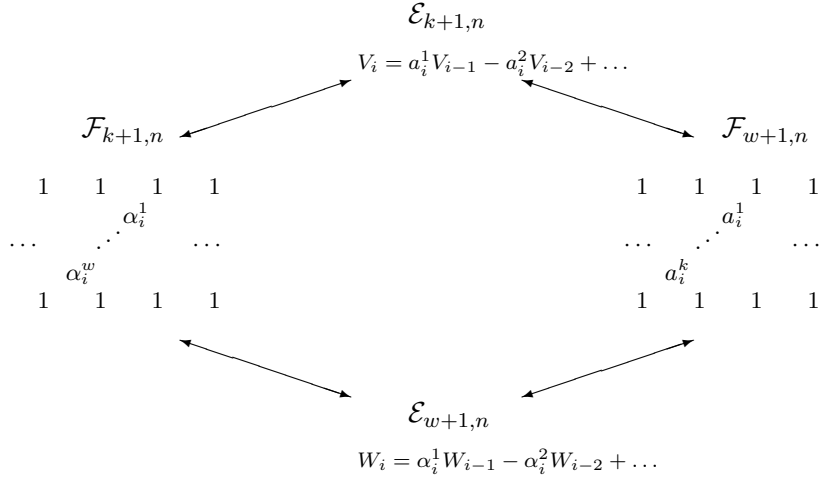
\begin{figure}
 \setlength{\unitlength}{3144sp}%
\begingroup\makeatletter\ifx\SetFigFont\undefined%
\gdef\SetFigFont#1#2#3#4#5{%
  \reset@font\fontsize{#1}{#2pt}%
  \fontfamily{#3}\fontseries{#4}\fontshape{#5}%
  \selectfont}%
\fi\endgroup%
\begin{picture}(6402,3817)(2251,-4111)
\put(5401,-511){\makebox(0,0)[lb]{\smash{{\SetFigFont{11}{13.2}{\rmdefault}{\mddefault}{\updefault}{\color[rgb]{0,0,0}$\Ec_{k+1,n}$}%
}}}}
\thinlines
{\color[rgb]{0,0,0}\put(6301,-961){\vector(-3, 1){  0}}
\put(6301,-961){\vector( 3,-1){1350}}
}%
{\color[rgb]{0,0,0}\put(3601,-3011){\vector(-3, 1){  0}}
\put(3601,-3011){\vector( 3,-1){1350}}
}%
{\color[rgb]{0,0,0}\put(4951,-961){\vector( 3, 1){  0}}
\put(4951,-961){\vector(-3,-1){1350}}
}%
\put(5000,-4011){\makebox(0,0)[lb]{\smash{{\SetFigFont{8}{13.2}{\rmdefault}{\mddefault}{\updefault}{\color[rgb]{0,0,0}$W_i=\alpha_i^1W_{i-1}-\alpha_i^2W_{i-2}+\ldots$}%
}}}}
\put(5401,-3661){\makebox(0,0)[lb]{\smash{{\SetFigFont{11}{13.2}{\rmdefault}{\mddefault}{\updefault}{\color[rgb]{0,0,0}$\Ec_{w+1,n}$}%
}}}}
\put(8426,-2311){\makebox(0,0)[lb]{\smash{{\SetFigFont{9}{13.2}{\rmdefault}{\mddefault}{\updefault}{\color[rgb]{0,0,0}$\cdots$}%
}}}}
\put(7101,-2311){\makebox(0,0)[lb]{\smash{{\SetFigFont{9}{13.2}{\rmdefault}{\mddefault}{\updefault}{\color[rgb]{0,0,0}$\cdots$}%
}}}}
\put(3701,-2311){\makebox(0,0)[lb]{\smash{{\SetFigFont{9}{13.2}{\rmdefault}{\mddefault}{\updefault}{\color[rgb]{0,0,0}$\cdots$}%
}}}}
\put(2251,-2311){\makebox(0,0)[lb]{\smash{{\SetFigFont{9}{13.2}{\rmdefault}{\mddefault}{\updefault}{\color[rgb]{0,0,0}$\cdots$}%
}}}}
\put(8551,-2761){\makebox(0,0)[lb]{\smash{{\SetFigFont{9}{13.2}{\rmdefault}{\mddefault}{\updefault}{\color[rgb]{0,0,0}1}%
}}}}
\put(8101,-2761){\makebox(0,0)[lb]{\smash{{\SetFigFont{9}{13.2}{\rmdefault}{\mddefault}{\updefault}{\color[rgb]{0,0,0}1}%
}}}}
\put(7651,-2761){\makebox(0,0)[lb]{\smash{{\SetFigFont{9}{13.2}{\rmdefault}{\mddefault}{\updefault}{\color[rgb]{0,0,0}1}%
}}}}
\put(7201,-2761){\makebox(0,0)[lb]{\smash{{\SetFigFont{9}{13.2}{\rmdefault}{\mddefault}{\updefault}{\color[rgb]{0,0,0}1}%
}}}}
\put(7426,-2536){\makebox(0,0)[lb]{\smash{{\SetFigFont{9}{13.2}{\rmdefault}{\mddefault}{\updefault}{\color[rgb]{0,0,0}$a_i^k$}%
}}}}
\put(7651,-2311){\makebox(0,0)[lb]{\smash{{\SetFigFont{9}{13.2}{\rmdefault}{\mddefault}{\updefault}{\color[rgb]{0,0,0}$\iddots$}%
}}}}
\put(7876,-2086){\makebox(0,0)[lb]{\smash{{\SetFigFont{9}{13.2}{\rmdefault}{\mddefault}{\updefault}{\color[rgb]{0,0,0}$a_i^1$}%
}}}}
\put(8551,-1861){\makebox(0,0)[lb]{\smash{{\SetFigFont{9}{13.2}{\rmdefault}{\mddefault}{\updefault}{\color[rgb]{0,0,0}1}%
}}}}
\put(8101,-1861){\makebox(0,0)[lb]{\smash{{\SetFigFont{9}{13.2}{\rmdefault}{\mddefault}{\updefault}{\color[rgb]{0,0,0}1}%
}}}}
\put(7651,-1861){\makebox(0,0)[lb]{\smash{{\SetFigFont{9}{13.2}{\rmdefault}{\mddefault}{\updefault}{\color[rgb]{0,0,0}1}%
}}}}
\put(7201,-1861){\makebox(0,0)[lb]{\smash{{\SetFigFont{9}{13.2}{\rmdefault}{\mddefault}{\updefault}{\color[rgb]{0,0,0}1}%
}}}}
\put(3826,-2761){\makebox(0,0)[lb]{\smash{{\SetFigFont{9}{13.2}{\rmdefault}{\mddefault}{\updefault}{\color[rgb]{0,0,0}1}%
}}}}
\put(3376,-2761){\makebox(0,0)[lb]{\smash{{\SetFigFont{9}{13.2}{\rmdefault}{\mddefault}{\updefault}{\color[rgb]{0,0,0}1}%
}}}}
\put(2926,-2761){\makebox(0,0)[lb]{\smash{{\SetFigFont{9}{13.2}{\rmdefault}{\mddefault}{\updefault}{\color[rgb]{0,0,0}1}%
}}}}
\put(2476,-2761){\makebox(0,0)[lb]{\smash{{\SetFigFont{9}{13.2}{\rmdefault}{\mddefault}{\updefault}{\color[rgb]{0,0,0}1}%
}}}}
\put(2701,-2536){\makebox(0,0)[lb]{\smash{{\SetFigFont{9}{13.2}{\rmdefault}{\mddefault}{\updefault}{\color[rgb]{0,0,0}$\alpha_i^w$}%
}}}}
\put(2926,-2311){\makebox(0,0)[lb]{\smash{{\SetFigFont{9}{13.2}{\rmdefault}{\mddefault}{\updefault}{\color[rgb]{0,0,0}$\iddots$}%
}}}}
\put(3151,-2086){\makebox(0,0)[lb]{\smash{{\SetFigFont{9}{13.2}{\rmdefault}{\mddefault}{\updefault}{\color[rgb]{0,0,0}$\alpha_i^1$}%
}}}}
\put(3826,-1861){\makebox(0,0)[lb]{\smash{{\SetFigFont{9}{13.2}{\rmdefault}{\mddefault}{\updefault}{\color[rgb]{0,0,0}1}%
}}}}
\put(3376,-1861){\makebox(0,0)[lb]{\smash{{\SetFigFont{9}{13.2}{\rmdefault}{\mddefault}{\updefault}{\color[rgb]{0,0,0}1}%
}}}}
\put(2926,-1861){\makebox(0,0)[lb]{\smash{{\SetFigFont{9}{13.2}{\rmdefault}{\mddefault}{\updefault}{\color[rgb]{0,0,0}1}%
}}}}
\put(2476,-1861){\makebox(0,0)[lb]{\smash{{\SetFigFont{9}{13.2}{\rmdefault}{\mddefault}{\updefault}{\color[rgb]{0,0,0}1}%
}}}}
\put(7876,-1411){\makebox(0,0)[lb]{\smash{{\SetFigFont{11}{13.2}{\rmdefault}{\mddefault}{\updefault}{\color[rgb]{0,0,0}$\Fc_{w+1,n}$}%
}}}}
\put(2826,-1411){\makebox(0,0)[lb]{\smash{{\SetFigFont{11}{13.2}{\rmdefault}{\mddefault}{\updefault}{\color[rgb]{0,0,0}$\Fc_{k+1,n}$}%
}}}}
\put(5000,-870){\makebox(0,0)[lb]{\smash{{\SetFigFont{8}{13.2}{\rmdefault}{\mddefault}{\updefault}{\color[rgb]{0,0,0}$V_i=a_i^1V_{i-1}-a_i^2V_{i-2}+\ldots$}%
}}}}
{\color[rgb]{0,0,0}\put(6301,-3461){\vector(-3,-1){  0}}
\put(6301,-3461){\vector( 3, 1){1350}}
}%
\end{picture}%
 \caption{Duality between periodic difference equations, friezes of solutions and friezes of coefficients.}
  \label{Illustre}
 \end{figure}
\end{center}

\subsection{Proof of Theorem \ref{TheGailThm} and Propositon~\ref{GDFprop}}\label{PruS}
$\,$

Let us consider the following $n$-(anti)periodic sequence~$(W_i)$.
On the interval $(W_{-w},\ldots,W_{n-w-1})$ of length $n$ we set:
$$
 \begin{array}{lccccccccccccccccccccccc}
(W_{-w},&\ldots,&W_{-1},&W_{0},&W_{1},&W_{2},&\ldots,
&W_{n-w-2},&W_{n-w-1}):=\\[6pt]
(0,&\ldots,&0,&1,&a_n^{k},&a_n^{k-1},&\ldots,&a_n^{1},&1)\\  \end{array}
$$
and then continue by (anti)periodicity: $W_{i+n}=(-1)^wW_i$.

\begin{lemma}
\label{SoLem}
The constructed sequence $(W_i)$ satisfies the equation~(\ref{TheDualEq}).
\end{lemma}

\noindent
{\it Proof of the lemma}.
By construction of the frieze pattern~(\ref{friezeGG}),
its South-East diagonals $\eta_i$ satisfy~\eqref{REq}.

 Consider the following selection of $k+2$ diagonals in the frieze pattern~(\ref{friezeGG}),
 and form the following $(k+2)\times{}n$-matrix:
$$
\begin{blockarray}{cccccccc}
\eta_{n-k-1} & \eta_{n-k} &  &  &  \eta_{n} \\[10pt]
\begin{block}{(ccccccc)c}
 \alpha^{1}_1 & 1 &0 & \cdots & 0 & \\[6pt]
\alpha^{2}_2 & \alpha^{1}_2 & 1 & 0 & \vdots& \\[6pt]
 \vdots & & \ddots & \ddots & 0 &\\[6pt]
 \vdots &  & &  \alpha^{1}_{k+1} &  1& \\[10pt]
\alpha^{w}_w &  &  &  &  \alpha^{1}_{k+2}&\\[6pt]
 1 &   \alpha^{w}_{w+1}& &  &  \vdots& \\[6pt]
0& \ddots & &  & \vdots &\\[6pt]
\vdots&& 1&  \alpha^{w}_{n-2}& \alpha^{w-1}_{n-2} &\\[6pt]
0&  & 0 & 1 & \alpha^{w}_{n-1} &\\[6pt]
 (-1)^k & 0& \cdots& 0 & 1 & \\[6pt]
\end{block}
\end{blockarray}
$$
The above diagonals satisfy the equation:
$$
\eta_n=a_n^1\eta_{n-1}-\ldots-(-1)^ka_n^k\eta_{n-k}+(-1)^k\eta_{n-k-1}.
$$
Let us express this equation for each component (i.e., each row of the above matrix). 

The first row gives $a_n^k=\alpha^{1}_1$, which can be rewritten as
$W_1=\alpha^{1}_1W_0$. 
Since, by construction, $W_{-1}=\cdots=W_{-w}=0$, we can rewrite this relation as
$$
  W_1= \alpha^{1}_1W_0-\alpha^{2}_1W_{-1}+\cdots+(-1)^wW_{-w},
$$
which is precisely~(\ref{TheDualEq}) for $i=1$.
Then considering the second component of the vectors $\eta$, we obtain the relation
$
0=a_n^{k-1}-a_n^k \alpha^{1}_2+ \alpha^{2}_2,
$
which reads as 
$$
  \begin{array}{cclllllllllllllllllllllllllllll}
 W_2&=& \alpha^{1}_2W_1&-& \alpha^{2}_2W_0&&\Longleftrightarrow\\[6pt]
 W_2&=& \alpha^{1}_2W_1&- &\alpha^{2}_2W_0&+
 &\alpha^{3}_2W_{-1}&+\cdots+&(-1)^wW_{1-w}.
 \end{array}
$$
Continuing this process, 
we obtain $n$ relations which correspond precisely to the equation~(\ref{TheDualEq}),
for $i=1,\ldots,n$.
Hence the lemma.
\hfill\proofend

Shifting the indices, in a similar way 
we obtain $w$ other solutions of the equation~(\ref{TheDualEq})
that, on the period $(W_{-w},\ldots,W_{n-w-1})$, are given by:
$$
 \begin{array}{cccccccccccccccccccccccc}
(0,&\ldots,&0,&1,&a_{n-1}^{k},&a_{n-1}^{k-1},&\ldots,&a_{n-1}^{1},&1,&0)\\[6pt]  
(0,&\ldots,&1,&a_{n-2}^{k},&a_{n-2}^{k-1},&\ldots,&a_{n-2}^{1},&1,&0,&0)\\ 
\vdots\\
(1,&a_{n-w}^{k},&a_{n-w}^{k-1},&\ldots,&a_{n-w}^{1},&1,&0,&\dots,&0,&0)
\end{array}
$$
Together with the solution from Lemma~\ref{SoLem}, these solutions
are linearly independent and therefore form a basis of $n$-(anti)periodic solutions
of equation~(\ref{TheDualEq}).
We proved that this equation indeed belongs to $\Ec_{w+1,n}$,
so that the map $\Gc$ is well-defined.

The relation between the equations~(\ref{REq}) and~(\ref{TheDualEq})
can be described as follows: the solutions of the former one are the
coefficients of the latter, and vice-versa.
Therefore, the map $\Gc$ is an involution.
This finishes the proof of
Theorem~\ref{TheGailThm}.
\hfill
\proofend

To prove Propositon~\ref{GDFprop}, it suffice
to notice that the diagonals of the pattern~\eqref{friezeG}
are exactly the solutions of the equation~(\ref{TheDualEq})
with initial conditions $(0,\ldots,0,1,a_i^{k})$.
This is exactly the way we associate a frieze pattern to a differential equation,
see Section~\ref{1stIso}.
\hfill
\proofend

\begin{example}
Let us give the most elementary, but perhaps the most striking, example
of the combinatorial Gale transform.
Suppose that a difference equation (\ref{REq}) of order~$k+1$ is such that
 all its solutions $(V_i)$
are $(k+3)$-(anti)periodic: 
$$
V_{i+k+3}=(-1)^k\,V_i.
$$
Consider the Hill equation
$$
W_{i}=a_{i}^1W_{i-1}-W_{i-2},
$$
obtained by ``forgetting'' the coefficients $a_{i}^{j}$ with $j\geq2$.
Theorem~\ref{TheGailThm} then implies that all the solutions of this
equation are antiperiodic with the same period:
$W_{i+k+3}=-W_i$.

Conversely, any difference equation (\ref{REq}) of order $k+1$
with $(k+3)$-(anti)periodic solutions can be constructed out of a Hill equation.
\end{example}

At first glance, it appears paradoxical that forgetting almost all the information
about the coefficients of a difference equation, we still keep the information about
the (anti)periodicity of solutions.

\subsection{Comparison with the classical Gale transform}\label{ClasS}

Recall that the classical Gale transform of configurations of points in projective spaces
is a map:
$$
\Gc_{\mathrm{class}}:{\mathcal C}_{k+1,n} 
\buildrel{\simeq}\over\longrightarrow{\mathcal C}_{w+1,n}\;,
$$ 
where $n=k+w+2$, see \cite{Gal,Cob,Cob1,EP}.

The classical Gale transform is defined as follows.
Let $A$ be a $(k+1)\times{}n$-matrix representing an element of ${\mathcal C}_{k+1,n}$, 
and $A'$ a $(w+1)\times{}n$-matrix representing an element of ${\mathcal C}_{w+1,n}$, 
see (\ref{Grass}) and (\ref{TheQuot}).
These elements are in Gale duality if there exists a
non-degenerate diagonal $n\times{}n$-matrix $D$
such that
$$
AD{A'}^T=0,
$$
where ${A'}^T$ is the transposed matrix.
This is precisely the duality of the corresponding Grassmannians combined
with the quotient (\ref{TheQuot}).

To understand the difference between the combinatorial Gale transform
and the classical Gale transform, recall that the space
${\mathcal E}_{k+1,n}\simeq{\mathcal F}_{k+1,n}$ 
is a subvariety of the Grassmannian~$\Gr_{k+1,n}$. 
Given an $\SL_{k+1}$-frieze pattern $F$,
the $(k+1)\times{}n$-matrix $M^{(i)}_F$ representing $F$, as in~(\ref{TheEmbM}),
satisfies the condition that every 
$(k+1)\times(k+1)$-minor 
$$
\Delta_I(M^{(i)}_F)=1.
$$
This implies that the diagonal matrix $D$ 
also has to be of a particular form.

\begin{proposition}
\label{FrGale}
Let $F$  be an $\SL_{k+1}$-frieze pattern and $\Gc(F)$ 
its Gale dual $\SL_{w+1}$-frieze pattern.
Then the corresponding matrices satisfy
\begin{equation}
\label{TheGaleM}
M^{(i)}_F\,D\,{M^{(j)}_{\Gc(F)}}^T=0,
\end{equation}
where $i-j=w+1\mod n$, and where the diagonal matrix 
\begin{equation}
\label{DDiagM}
D=\left(
\begin{array}{cccc}
1&0&\ldots&0\\[4pt]
0&-1&\ldots&0\\
\vdots&\vdots&\ddots&\vdots\\[4pt]
0&0&\ldots&(-1)^{n-1}
\end{array}
\right).
\end{equation}
\end{proposition}

\begin{proof}
The matrices are explicitly given by 
$$
M^{(i)}_F=
\left(
\begin{array}{cccccccccccccc}
 1 &   \alpha^{w}_{i-1}& &  &   &&\alpha^{1}_{i-1}&1 \\[4pt]
& \ddots & &  & &&&&\ddots\\[4pt]
&    & 1 & \alpha^{w}_{i+k-1} &&&&&\alpha^{1}_{i+k-1} &1\\[6pt]
\end{array}\right)\;,
$$
of size $(k+1)\times  n$ and
$$
M^{(j)}_{\Gc(F)}=
\left(
\begin{array}{cccccccccccccc}
 1 &   a^{k}_{j-1}& &     &&a^{1}_{j-1}&1 \\[4pt]
 & 1&  a^{k}_{j}& &&&a^{1}_{j}&1\\[2pt]
&& \ddots & &  & &&&&\ddots\\[4pt]
&  &  & 1 & a^{k}_{j+w-1} &&&&&a^{1}_{j+w-1} &1\\[6pt]
\end{array}\right)\;,
$$
of size $(w+1)\times  n$.
The columns of the matrix $M^{(i)}_F$
correspond to the diagonals $\eta_{i-1}, \cdots, \eta_{i+k+w}$ in $F$, which are solutions of \eqref{REq}.
This gives immediately the relation
$$
M^{(i)}_F\,D\,{M^{(j)}_{\Gc(F)}}^T=0,
\quad  \text{ where } \quad 
  D=\mathrm{diag}(1,-1,1,-1,\ldots ).
$$
\end{proof}

Recall that in Section~\ref{1stIso}, we  defined a projection
from the space of equations to the moduli space of $n$-gons.
The Gale transform agrees with this projection.

\begin{corollary}
 The following diagram commutes
 $$
   \xymatrix{
    \Ec_{k+1,n} \ar[r]^{\Gc} \ar[d]& \Ec_{w+1,n} \ar[d]\\
 \Cc_{k+1,n}  \ar[r]_{\Gc_{\mathrm{class}}} &\Cc_{w+1,n}  
  }
 $$
\end{corollary}

\begin{proof}
The projection $\Ec_{k+1,n}\to\Cc_{k+1,n}$, written in the matrix form,
associates to a matrix representing an element of $\Ec_{k+1,n}$ a coset
in the quotient $\Gr_{k+1,n}\slash\T^{n-1}$ defined by
left multiplication by diagonal $n\times{}n$-matrices.
If two representatives of the coset satisfy~(\ref{TheGaleM}),
then any two other representatives also do.
\end{proof}

$\,$

 \begin{example}
 Choose $k=2, w=3, n=7$.
 The two friezes of Figure \ref{DualFriezes} are Gale dual to each other.
 
\begin{figure}[hbtp]
\input{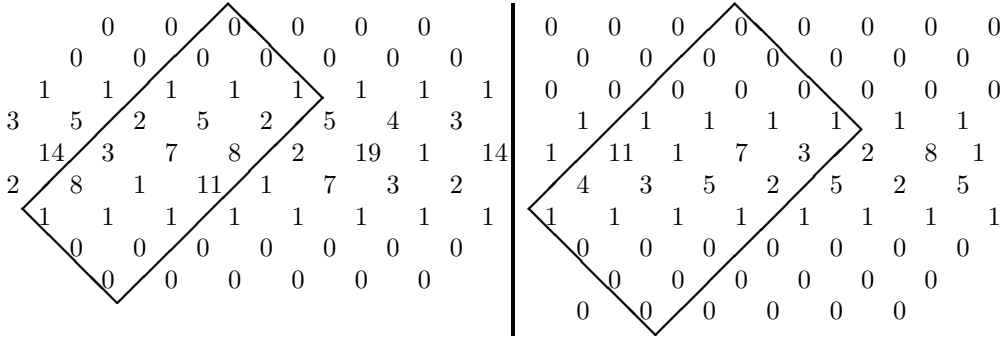}
\caption{An $\SL_3$-frieze of width 3 (matrix $A$ left) and its Gale dual 
$\SL_4$-frieze of width 2 (matrix $A'$ right).}
\label{DualFriezes}
\end{figure}

\noindent
One immediately checks in this example that the corresponding matrices satisfy
 $$
 ADA'^T=0,\quad  \text{ for } \quad 
  D=\mathrm{diag}(1,-1,1,-1,1,-1,1). 
  $$
 \end{example}

\subsection{The projective duality}\label{DualSect}
Recall that the dual projective space $(\RP^k)^*$
(which is of course itself isomorphic to $\RP^k$) is the space
of hyperplanes in~$\RP^k$.
The notion of {\it projective duality} is central in projective geometry.

Projective duality is usually defined for generic $n$-gons as follows.
Given an $n$-gon $(v_i)$ in $\RP^k$, the {\it projectively dual $n$-gon}
$(v_i^*)$ in $(\RP^k)^*$ is the $n$-gon such that each vertex $v_i^*$
is the hyperplane $(v_i , v_{i+1},\ldots,v_{i+k-1})\subset\RP^k$.
This procedure obviously commutes with the action of $\SL_{k+1}$,
so that one obtains a map 
$$
*:{\mathcal C}_{k+1,n}\to{\mathcal C}_{k+1,n},
$$
which squares to a shift: $*\circ*:(v_i)\mapsto(v_{i+k-1})$.

In this section, we introduce an
analog of the projective duality on the space of difference equations
and that of frieze patterns:
$$
*:{\mathcal E}_{k+1,n}\to{\mathcal E}_{k+1,n},
\qquad
*:{\mathcal F}_{k+1,n}\to{\mathcal F}_{k+1,n}.
$$
The square of $*$ also shifts the indices, but this shift is
``invisible'' on equations and friezes so that it is an  {\it involution}: 
$*\circ*=\mathrm{id}$.

\begin{definition}
The difference equation
\begin{equation}
\label{DualREq}
V^*_{i}=a_{i+k-1}^kV^*_{i-1}-a_{i+k-2}^{k-1}V^*_{i-2}+ 
\cdots+(-1)^{k-1}a_{i}^1V^*_{i-k}+(-1)^{k}V^*_{i-k-1}
\end{equation}
is called the {\it projective dual of equation} (\ref{REq}).
\end{definition}

Here $(V_i^*)$ is just a notation for the unknown.

\begin{example}
\label{DeEx}
The projective dual of the equation
$V_{i}=a_{i}V_{i-1}-b_iV_{i-2}+V_{i-3}$ is:
$$
V^*_{i}=b_{i+1}V^*_{i-1}-a_iV^*_{i-2}+V^*_{i-3}.
$$
\end{example}

The above definition is justified by the following statement.

\begin{proposition}
\label{DualProp}
The map ${\mathcal E}_{k+1,n}\to{\mathcal C}_{k+1,n}$
from Section~\ref{1stIso} commutes with projective duality.
\end{proposition}

\begin{proof}
Recall that an $n$-gon $(v_i)$ is in the image of the map 
${\mathcal E}_{k+1,n}\to{\mathcal C}_{k+1,n}$
if and only if it is a projection of an $n$-gon $(V_i)$ in
$\R^{k+1}$ satisfying the determinant condition (\ref{detconst}).

Let us first show that
the dual $n$-gon $(v_i^*)$ is also in the image of 
the map ${\mathcal E}_{k+1,n}\to{\mathcal C}_{k+1,n}$.
Indeed, by definition of projective duality,
the affine coordinates of a vertex $v_i^*\in(\RP^k)^*$ of the dual $n$-gon
can be calculated as the $k\times{}k$-minors of
the $k\times(k+1)$-matrix
$$
\left(
V_i \, V_{i+1}\,\ldots\,V_{i+k-1}
\right)
$$
where $V_j$ are understood as $(k+1)$-vectors (i.e., the columns of the matrix).
Denote by $V^*_i$ the vector in $(\R^{k+1})^*$ with coordinates
given by the $k\times{}k$-minors.
In other words, the vector $V^*_i$ is defined by the equation
$$
\left|
V_i , V_{i+1},\ldots,V_{i+k-1},V_i^*
\right|=1.
$$

A direct verification then shows that $(V_i^*)$ 
satisfy the equation (\ref{DualREq}).
\end{proof}

The isomorphism ${\mathcal E}_{k+1,n}\simeq{\mathcal F}_{k+1,n}$ 
allows us to define the notion of projective
duality on $\SL_{k+1}$-frieze patterns.

\begin{proposition}
\label{Turn}
The projective duality of $\SL_{k+1}$-frieze patterns is 
just the symmetry with respect to the median horizontal axis.
\end{proposition}

We will prove this statement in Section~\ref{DeTSTwo}.
The proof uses the explicit computations. See Proposition~\ref{DualDiag}.

\begin{corollary}
The projective duality commutes with the Gale transform:
$$
*\circ\Gc=\Gc\circ*.
$$
\end{corollary}

\subsection{The self-dual case}\label{SDS}

An interesting class of difference equations
and equivalently, of $\SL_{k+1}$-frieze patterns, is the class of
{\it self dual} equations.
In the case of frieze patterns, self-duality means invariance with
respect to the horizontal axis of symmetry.

\begin{example}
a) Every $\SL_{2}$-frieze pattern is self-dual.

b) Consider the following $\SL_{3}$-frieze patterns of width $2$:
$$
\begin{array}{rrrrrrrr}
1&&1&&1&&1&\\
&2&&2&&2&&2\\
2&&2&&2&&2&\\
&1&&1&&1&&1
\end{array}
\qquad\qquad
\begin{array}{rrrrrrrr}
1&&1&&1&&1&\\
&2&&3&&2&&3\\
1&&5&&1&&5&\\
&1&&1&&1&&1
\end{array}
$$
The first one is self-dual but the second one is not.
\end{example}

An $n$-gon is called {\it projectively self dual} if, for some fixed $0\leq\ell\leq{}n-1$,
the $n$-gon $(v_{i+\ell}^*)$ is projectively equivalent to $(v_i)$.
Note that $\ell$ is a parameter in the definition (so, more accurately, one should say 
``$\ell$-self dual''); see~\cite{FT}.

\section{The determinantal formulas}\label{DeTS}

In this section, we give explicit formulas for the Gale transform.
It turns out that one can solve the equation (\ref{REq})
and obtain explicit formulas for the coefficients of the $\SL_{k+1}$-frieze pattern.
Let us mention that the determinant formulas presented 
here already appeared in the classical literature on difference equations
in the context of ``Andr\'e method of solving difference equations'';
see \cite{And,Jor}.

\subsection{Calculating the entries of the frieze patterns}\label{DeTSOne}

Recall that we constructed an isomorphism between
the spaces of difference equations ${\mathcal  E}_{k+1,n}$ 
and frieze patterns ${\mathcal  F}_{k+1,n}$.
We associated an $\SL_{k+1}$-frieze pattern of width $w$
to every difference equation (\ref{REq}).
The entries~$d_{i,i+j}$ of the $\SL_{k+1}$-frieze pattern 
(also denoted by $\alpha_{i-1}^{w-j}$, see~(\ref{friezeGG}), (\ref{Corresp})) 
were defined non-explicitly by (\ref{ConstEq}).

\begin{proposition}
\label{DetExpr}
The entries of the $\SL_{k+1}$-frieze pattern 
associated to a difference equation~(\ref{REq}) are expressed 
in terms of the coefficients $a_i^j$
by the following determinants.

(i)
If $0\leq j \leq k-1$ and $j<w$, then
\begin{equation}
\label{DetEq}
d_{i,i+j}=
\left| 
\begin{array}{llllll}
a_{i}^1&1&\\[8pt]
a_{i+1}^2&a_{i+1}^1&1&\\[8pt]
\vdots&\ddots&\ddots&\;1\\[6pt]
a_{i+j}^{j+1}&\cdots&a_{i+j}^{2}&a_{i+j}^{1}
\end{array}
\right|.
\end{equation}

(ii)
If $k-1< j <w$, then
\begin{equation}
\label{DetEqDva}
d_{i,i+j}=
\left| 
\begin{array}{cccclcc}
a_{i}^1&1&\\[8pt]
\vdots&a_{i+1}^1&1&\\[4pt]
a_{i+k-1}^{k}&&\ddots&\ddots\\[6pt]
1&&&\ddots&\ddots\\[4pt]
&\ddots&&&a_{i+j-1}^1&1\\[8pt]
&&1&a_{i+j}^{k}&\ldots&a_{i+j}^1
\end{array}
\right|.
\end{equation}
\end{proposition}

\begin{proof}
We use the  Gale transform $\Gc(F)$ of the frieze $F$ associated to \eqref{REq}.
In $\Gc(F)$, the following diagonals of length $w+1$
$$
\begin{blockarray}{cccccccccccccc}
\eta_{i-w-2} & \eta_{i-w-1} &  &    \eta_{i-w+j}&&\eta_{i-2}& \eta_{i-1} & \\[10pt]
\begin{block}{(ccccccccccccc)c}
a_{i}^1&1&\\[8pt]
a_{i+1}^2&a_{i+1}^1&\ddots&\\[8pt]
\vdots&\vdots&&\;1\\[6pt]
a_{i+j}^{j+1}&a_{i+j}^{j}&&a_{i+j}^{1}&\ddots\\
\vdots&\vdots& &  \vdots& &1&\\[4pt]
\vdots&\vdots&&\vdots&&a_{i+w}^1&1\\[6pt]
\end{block}
\end{blockarray}\;,
$$
satisfy the recurrence relation 
$$
\eta_{i-1}=\alpha_{i-1}^1\eta_{i-2}-\ldots +(-1)^{w-j-1}\alpha_{i-1}^{w-j}\eta_{i-w+j-1}+\ldots
 +(-1)^{w-1}\ \alpha_{i-1}^{w}\eta_{i-w-1} +(-1)^{w}\ \eta_{i-w-2},
$$
that can be written in terms of vectors and matrices as
$$
\eta_{i-1}=
\left(
\begin{array}{ccc}
\eta_{i-w-2}, &\ldots, & \eta_{i-2}
\end{array}\right)
\left(
\begin{array}{l}
(-1)^{w}\\
(-1)^{w-1}\alpha_{i-1}^w\\
\vdots\\
(-1)^0\alpha_{i-1}^1
\end{array}\right).
$$
The coefficients $\alpha_{i-1}^{w-j}$ can be computed using the Cramer rule
\begin{equation}
\label{CramEq}
\alpha_{i-1}^{w-j}=\dfrac{(-1)^{w-j-1}|\eta_{i-w-2}, \ldots, \eta_{i-w+j}, \;
\eta_{i-1}, \;\eta_{i-w+j+2}, \ldots, \eta_{i-2}|}{|\eta_{i-w-2}, \ldots,  \eta_{i-2}|}.
\end{equation}
The denominator is 1 since $\Gc(F)$ is an $\SL_{w+1}$-frieze, and the numerator simplifies to 
\eqref{DetEq} or to~\eqref{DetEqDva} accordingly after decomposing by the $\eta_{i-1}$-th column. 
The coefficient $\alpha_{i-1}^{w-j}$ is in position~$d_{i,i+j}$ in the frieze $F$.
\end{proof}

\subsection{Equivalent formulas}\label{DeTSTri}

There is another, alternative, way to calculate the entries
of the $\SL_{k+1}$-frieze pattern.

\begin{proposition}
\label{DetExprDva}
(i) If $j+k\geq w$ then
$$
d_{i,i+j}=
\left| 
\begin{array}{ccrrc}
a_{i-w+j-1}^k&a_{i-w+j-1}^{k-1}&\ldots&\ldots&a_{i-w+j-1}^{k-w+j+1}\\[8pt]
1&a_{i-w+j}^k&\ldots&\ldots&a_{i-w+j}^{k-w+j}\\[8pt]
&1&&&\vdots\\[8pt]
&&\ddots&\ddots&\vdots\\[12pt]
&&&\;1&a_{i-2}^k\\[8pt]
\end{array}
\right|.
$$

(ii)
If $j+k<w$ then
\begin{equation*}
\label{DetEqBis}
d_{i,i+j}=
\left| 
\begin{array}{cccccc}
a_{i-w+j-1}^k&\ldots&a_{i-w+j-1}^1&\;\;1&&\\[8pt]
1&a_{i-w+j}^k&\ldots&a_{i-w+j}^1&\!\!1&\\[8pt]
&1&\ldots&\ldots&&\!\!1\\[8pt]
&&\ddots&&\ddots&\vdots\\[12pt]
&&&\;1&a_{i-3}^k&a_{i-3}^{k-1}\\[8pt]
&&&&\!\!\!\!1&a_{i-2}^k
\end{array}
\right|.
\end{equation*}
\end{proposition}

\begin{proof}
These formulas are obtained in the same way as~(\ref{DetEq}) and (\ref{DetEqDva}).
\end{proof}

\begin{example}
\label{HillEx2}
{\rm
(a)
Hill's equations $V_{i}=a_{i}V_{i-1}-V_{i-2}$ with antiperiodic solutions correspond to
Coxeter's  frieze patterns with the entries
$$
d_{i,i+j}=
\left|
\begin{array}{cccccc}
a_{i}&1&&&\\[4pt]
1&a_{i+1}&1&&\\[4pt]
&\ddots&\ddots&\ddots&\\[4pt]
&&1&a_{i+j-2}&1\\[4pt]
&&&1&a_{i+j-1}
\end{array}\right|.
$$
The corresponding geometric space is the moduli space (see Theorem \ref{TriThm})
of $n$-gons in the projective line known (in the complex case) 
under the name of moduli space $\mathcal{M}_{0,n}$.

(b)
In the case of third-order difference equations,
$V_{i}=a_{i}V_{i-1}-b_iV_{i-2}+V_{i-3}$,
we have:
$$
d_{i,i+j}=
\left|
\begin{array}{llllll}
a_{i}&b_{i+1}&1&&&\\[4pt]
1&a_{i+1}&b_{i+2}&1&&\\[4pt]
&\;\ddots&\; \ddots& \;\ddots&\; \ddots&\\[4pt]
&&1&a_{i+j-3}&b_{i+j-3}&1\\[4pt]
&&&1&a_{i+j-2}&b_{i+j-2}\\[4pt]
&&&&1&a_{i+j-1}
\end{array}
\right|.
$$
This case is related to the moduli space $\Cc_{3,n}$ of $n$-gons
in the projective plane studied in~\cite{MGOT}.
}
\end{example}

\subsection{Determinantal formulas for the Gale transform}\label{DeTSFour}

Formulas from Propositions~\ref{DetExpr} and~\ref{DetExprDva} 
express the entries of the $\SL_{k+1}$-frieze pattern~(\ref{friezeGG})
as minors of the Gale dual $\SL_{w+1}$-frieze pattern~(\ref{friezeG}).
One can reverse the formula and express the entries of~(\ref{friezeG}),
i.e., the coefficients of the equation~(\ref{REq}) as minors of
the $\SL_{k+1}$-frieze pattern~(\ref{friezeGG}). 

For instance the Gale dual of~\eqref{CramEq}
is
\begin{equation}
\label{DuDeT}
a_{i-1}^{k-j}=
\dfrac{(-1)^{k-j-1}|\eta_{i-k-2}, \ldots, \eta_{i-k+j}, \;
\eta_{i-1}, \;\eta_{i-k+j+2}, \ldots, \eta_{i-2}|}{|\eta_{i-k-2}, \ldots,  \eta_{i-2}|},
\end{equation}
where $\eta$'s are the South-East diagonals of the $\SL_{k+1}$-frieze pattern~\eqref{friezeGG}.
Note that the denominator is equal to $1$.

For $j<k$ we obtain an analog of the formula~\eqref{DetEq}:
\begin{equation}
\label{DuDeTBis}
a_{i-1}^{k-j}=
\left| 
\begin{array}{cccc}
d_{i+1,i+w}&1&\\[8pt]
\vdots&\ddots&\;1\\[6pt]
d_{i+j+1,i+w}&\cdots&d_{i+j+1,i+j+w}
\end{array}
\right|,
\end{equation}
and similarly for~\eqref{DetEqDva}.

\subsection{Frieze patterns and the dual equations}\label{DeTSTwo}

\begin{proposition}
\label{DualDiag}
The North-East diagonals $(\mu_i)$ of
the $\SL_{k+1}$-frieze pattern corresponding to
the difference equation (\ref{REq}) satisfy the projectively dual
difference equation~(\ref{DualREq}).
\end{proposition}

\begin{proof}
This is a consequence of the tameness of the $\SL_{k+1}$-frieze pattern corresponding to \eqref{REq}
and can be easily showed using the explicit determinantal formulas \eqref{DuDeT} that will be 
 established in the next section.
We  first consider the case $j=i+w-k-2$.
We know that $(k+2)\times (k+2)$-determinant vanishes, thus
$$
\left| 
\begin{array}{ccccc}
d_{j-w+1,j}&1&\\[8pt]
d_{j-w+2,j}&d_{j-w+2,j+1}&1\\[12pt]
\vdots&&\ddots&\;1\\[6pt]
d_{i,j}&\cdots&&d_{i,i+w-1}
\end{array}
\right|=0.
$$
Decomposing the determinant by the first column gives the recurrence relation
$$
d_{i,j}=a_{i-2}^kd_{i-1,j}-a_{i-3}^{k-1}d_{i-1,j}+ 
\cdots+(-1)^{k-1}a_{i-k-1}^1d_{i-k,j}+(-1)^{k}d_{i-k-1,j},
$$
where the coefficients are computed using \eqref{DuDeT}.
The recurrence relation then propagates inside the frieze due to the tameness property, and 
hence will hold for all $j$. 
This proves that the North-East diagonals satisfy \eqref{DualREq} after renumbering $\mu'_i:=\mu_{i+k+1}$.
\end{proof}

\section{The Gale transform and the representation theory}\label{RePSS}

In this section, we give another description of the combinatorial Gale transform
$\Gc$ in terms of representation theory of the Lie group $\SL_{n}$.
Let $N\subset\SL_{n}$ be the subgroup of upper unitriangular matrices.
$$
A=\left(
\begin{array}{rccl}
1&*&\ldots&*\\
&\ddots&\ddots&\vdots\\
&&\ddots&*\\
&&&1
\end{array}
\right).
$$ 
We will associate a unitriangular $n\times{}n$-matrix 
to every $\SL_{k+1}$-frieze pattern.
This idea allows us to apply in our situation many tools of the theory of matrices,
as well as more sophisticated tools of representation theory.
We make just one small step in this direction: 
we show that the combinatorial Gale transform $\Gc$ 
coincides with the restriction of the anti-involution on $N$ 
introduced in~\cite{BFZ}.

\subsection{From frieze patterns to unitriangular matrices}

Given an $\SL_{k+1}$-frieze pattern $F$ of width $w$,
as in formula~(\ref{FREq}),
cutting out a piece of the frieze,
see Figure~\ref{CutFig},
we associate to $F$ a unitriangular $n\times n$-matrix 
\begin{equation}
\label{TheMA}
A_F=\left(
\begin{array}{rcccccc}
1&d_{1,1}&\cdots&d_{1,w}&1&&\\
&\ddots&\ddots&&\ddots&\ddots&\\
&&\ddots&\ddots&&\ddots&1\\
&&&\ddots&\ddots&&d_{n-w,n-1}\\
&&&&\ddots&\ddots&\vdots\\
&&&&&\ddots&d_{n-1,n-1}\\[4pt]
&&&&&&1
\end{array}
\right)
\end{equation}
with $w+2$ non-zero diagonals.
As before, $n=k+w+2$.
\begin{figure}[hbtp]
\input{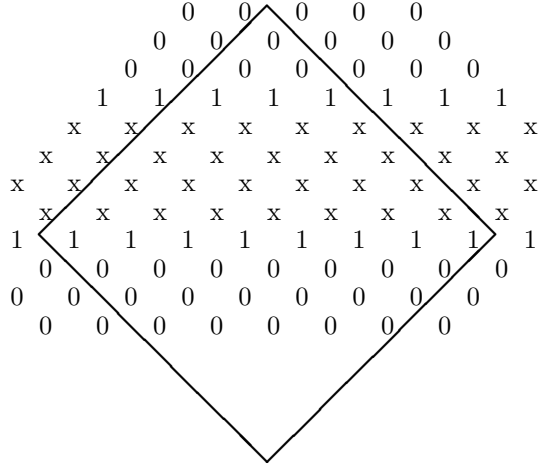}
\caption{Cutting a unitriangular matrix out of a frieze.}
\label{CutFig}
\end{figure}

\begin{remark}
Note that the matrix $A_F$ contains all the
 information about the $\SL_{k+1}$-frieze pattern.
Indeed, the frieze pattern can be reconstructed by even a much smaller 
$(k+1)\times{}n$-matrix (\ref{TheEmbM}).
However, $A_F$ is not defined uniquely.
Indeed, it depends on the choice of the first element~$d_{1,1}$
in the first line of the frieze.
A different cutting gives a different matrix.
\end{remark}

\subsection{The combinatorial Gale transform as an anti-involution on $N$}
Denote by $^{\iota}$ the anti-involution of $N$ defined for $x\in{}N$ by 
\begin{equation}
\label{IoT}
x^\iota=Dx^{-1}D,
\end{equation}
where $D$ is the diagonal matrix~(\ref{DDiagM}).
Recall that the term ``anti-involution'' means an involution that is an anti-homomorphism,
i.e., $(xy)^\iota=y^\iota{}x^\iota$.
This anti-involution was introduced in~\cite{BFZ} in order to study
the canonical parametrizations of $N$.
We will explain the relation of $^\iota$ to the classical representation theory in Section~\ref{ClSec}. 

\begin{theorem}
(i)
The operation $A_F\mapsto (A_F)^{\iota}$
associates to a matrix~(\ref{TheMA}) of a tame $\SL_{k+1}$-frieze pattern
a matrix of a tame $\SL_{w+1}$-frieze pattern.

(ii)
The corresponding map 
$$
^\iota:\Ec_{k+1,n}\to\Ec_{w+1,n}
$$
coincides with the composition of the Gale transform and
the projective duality: 
$$
^\iota=\Gc\circ*.
$$
\end{theorem}

\begin{proof}
Let us denote by $S=\{1, 2,\ldots, n\}$ the index set of the rows, resp. columns, 
of a matrix~$x\in N$.
For two subsets  $I,J\subset{}S$ of the same cardinality we denote by 
$\Delta_{I,J}(x)$ the minor of the matrix  $x$ 
taken over the rows of indices in $I$ and
the columns of indices in $J$.
We have the following well-known relation
$$
\Delta_{i,j}(x^{\iota})=\Delta_{S-\{ j\}, S-\{i\}}(x),
$$
where $\Delta_{i,j}(x^{\iota})$ is simply the entry in position $(i,j)$ in $x^{\iota}$.
Taking into  account that the matrix $x$ belongs to $N$, 
whenever $j> i$, the $(n-1)\times (n-1)$-minor in the above right hand side 
simplifies to a $(j-i)\times (j-i)$-minor
$$
\Delta_{i,j}(x^{\iota})=\Delta_{[ i,j-1], [i+1,j]}(x),
$$
where $[a,b]$ denotes the interval $\{a,a+1,\ldots,b\}$.
Let us use this relation to compute the entry in position $(i,j)$ in $(A_F)^{\iota}$.
One has:
$$
\Delta_{i,j}(A_F^{\iota})=
\left| 
\begin{array}{ccrrc}
d_{i,i}&d_{i,i+1}&\ldots&\ldots&d_{i,j-1}\\[8pt]
1&d_{i+1,i+1}&\ldots&\ldots&d_{i+1,j-1}\\[8pt]
&1&&&\vdots\\[8pt]
&&\ddots&\ddots&\vdots\\[12pt]
&&&\;1&d_{j-1,j-1}\\[8pt]
\end{array}
\right|=
\left| 
\begin{array}{ccrrc}
\alpha_{i-1}^{w}&\alpha_{i-1}^{w-1}&\ldots&\ldots&\alpha_{i-1}^{w-j+i-1}\\[8pt]
1&\alpha_{i}^w&\ldots&\ldots&\alpha_{i}^{w-j+i}\\[8pt]
&1&&&\vdots\\[8pt]
&&\ddots&\ddots&\vdots\\[12pt]
&&&\;1&\alpha_{j-2}^w\\[8pt]
\end{array}
\right|.
$$
According to the determinantal formulas of Section~\ref{DeTS}, 
this is precisely the entry 
$d_{j,i+k}$ of the frieze $\Gc(F)$,
and hence the result.
\end{proof}

\subsection{Elements of the representation theory}

Our next goal is to explain the relation of the involution $^\iota$
with Schubert cells in the Grassmannians.
We have to recall some basic notions of representation theory.

The {\it Weyl group} of $\SL_{n}$ is the group $\Sc_{n}$ 
of permutations over $n$ letters
that we think of as the set of integers  $\{1, 2, \ldots, n\}$.
The group $\Sc_{n}$ is generated by $(n-1)$ elements denoted 
by $s_i$, $1\leq i\leq n-1$, representing the elementary transposition
$i \leftrightarrow i+1$.
The relations between the generators are as follows:
$$
\begin{array}{rcl}
s_{i}s_{i+1}s_{i}&=&s_{i+1}s_{i}s_{i+1},\\[4pt]  
s_is_j&=&s_js_i, \quad  |i-j|> 1.
\end{array}
$$
A decomposition of $\s\in\Sc_{n}$ as
$$
\s=s_{i_1}s_{i_2}\cdots s_{i_p}
$$
is called \textit{reduced} if it involves the least possible number of generators. 
Equivalently, we call the sequence $\bfi=(i_1, \ldots, i_p)$ a \textit{reduced word} for $\s$.

The group $\Sc_{n}$ can be viewed as a subgroup of $\SL_{n}$ using the following lift of the generators
$$
s_i=\begin{pmatrix}
\ddots&&&\\
&0&1&\\
&-1&0&&\\
&&&\ddots
\end{pmatrix}.
$$

Let us now describe the {\it standard parametrization}
of the unipotent subgroup $N$.
Consider the following one-parameter subgroups of $N$:
$$
x_i(t)=
\begin{pmatrix}
\ddots&&&\\
&1&t&\\
&&1&&\\
&&&\ddots
\end{pmatrix}, \quad
1\leq i \leq n-1, 
\quad 
t\in\R,
$$
where $t$ is in position $(i,i+1)$.
The matrices $x_i(t)$ are called the elementary Jacobi matrices,
these are generators of $N$.

The next notion we need is that of the {\it Schubert cells}.
Denote by $B^-$ the Borel subgroup of lower triangular matrices of $\SL_{n}$. 
Fix an arbitrary element $\s\in\Sc_{n}$, 
and consider the set 
$$
N^\s:=N\cap B^-\s{}B^-
$$ 
which is known as
an open dense subset of a Schubert cell.
It is well known that generically elements of $N^\s$ can be represented as 
\begin{equation}
\label{ThexEq}
x_\bfi (\bft) =x_{i_1}(t_1)x_{i_2}(t_2)\cdots x_{i_p}(t_p),
\end{equation}
where $\bft=(t_1,\ldots, t_p)\in \R^p$ and 
$\bfi=(i_1,\ldots, i_p)$ is an arbitrary reduced word for $\s$. 

For the following choice 
\begin{equation}
\label{Thesigma}
\s=s_{k+1}\cdots s_{n-1}\, s_k\cdots s_{n-2}\, \cdots  \, s_{1}\cdots s_{n-k-1},
\end{equation}
the set $N^\s$ is identified with an open subset of the Grassmannian $\Gr_{k+1, n}$.
See, e.g.,~\cite[Chap.~8]{Spr}, for more details.

\subsection{The anti-involution $^{\iota}$ and the Grassmannians}\label{ClSec}

It was proved in~\cite{BFZ} that
the anti-involution $^{\iota}$ on $N$ can be written 
in terms of the generators as follows.
Set $x_i(t)^{\iota}=x_i(t)$ and for $x$ as in~(\ref{ThexEq}),
one has:
$$
x^\iota=
x_{i_p}(t_p)x_{i_{p-1}}(t_{p-1})\cdots{}x_{i_1}(t_1).
$$
This map is well-defined, i.e., it is independent of the choice
of the decomposition of $x$ into a product of generators since it coincides with~(\ref{IoT}).

Restricted to $N^\s$, where $\s$ is given by~(\ref{Thesigma}), the anti-involution reads:
$
^{\iota}:N^\s\to{}N^{\s^{-1}}.
$
In particular, it sends an open subset of the Grassmannian $\Gr_{k+1, n}$
to an open subset of~$\Gr_{w+1, n}$.

We have already defined the embedding (\ref{TheEmbM}) of the space of
$\SL_{k+1}$-frieze patterns into the Grassmannian.
Quite obviously, one also has an embedding into $N^\s$, so that
$$
{\mathcal F}_{k+1,n}\subset{}N^\s\subset\Gr_{k+1,n}.
$$
It can be shown that the image of the involution
$^{\iota}$ restricted to ${\mathcal F}_{k+1,n}$
belongs to ${\mathcal F}_{w+1,n}$.
However, the proof is technically involved and we do not dwell on the details here.
The following example illustrates the situation quite well.

\begin{example}
The case of $\Gr_{2,5}$.
Fix $\s=s_2s_3s_4s_1s_2s_3$ and consider the following element of~$N^\s$:
$$
x=x_2(t_1)x_3(t_2)x_4(t_3)x_1(t_4)x_2(t_5)x_3(t_6)=
\begin{pmatrix}
1 & t_4& t_4t_5&t_4t_5t_6&0\\[4pt]
&1&t_5+t_1&t_1t_2+t_1t_6+t_5t_6&t_1t_2t_3\\[4pt]
&&1&t_6+t_2&t_2t_3\\[4pt]
&&&1&t_3\\[4pt]
&&&&1
\end{pmatrix}.
$$
One then has
$$
x^{\iota}=x_3(t_6)x_2(t_5)x_1(t_4)x_4(t_3)x_3(t_2)x_2(t_1)=
\begin{pmatrix}
1 & t_4& t_4t_1&0&0\\[4pt]
&1&t_5+t_1&t_2t_5&0\\[4pt]
&&1&t_6+t_2&t_6t_3\\[4pt]
&&&1&t_3\\[4pt]
&&&&1
\end{pmatrix}.
$$
If now $x\in\Fc_{2,5}$, so that every $2\times2$-minor equals 1,
then one has after an easy computation:
$$
t_1t_4=1,
\quad
t_1t_2t_4t_5=1,
\quad
t_1t_2t_3t_4t_5t_6=1,
\quad
t_1t_2t_3=1,
$$
and this implies that $x^{\iota}\in\Fc_{3,5}$.
\end{example}

\section{Periodic rational maps from frieze patterns}\label{PerSec}

The periodic rational maps described in this section 
are a simple consequence of the isomorphism
$\Ec_{k+1,n}\simeq\Fc_{k+1,n}$ and of periodicity condition.
However, the maps are of interest.
The simplest example is known as the {\it Gauss map}, see~\cite{Gau}.
This map is given explicitly by
$$
(c_1,c_2)\mapsto\left(c_2,\frac{1+c_1}{c_1c_2-1}\right),
$$ 
where $c_1,c_2$ are variables.
Gauss proved that this map is $5$-periodic.
In our terminology, 
Gauss' map consists in the index shift: $(c_1,c_2)\mapsto(c_2,c_3)$
in a Hill equation $V_i=c_iV_{i-1}-V_{i-2}$ with $5$-antiperiodic solutions.

The maps that we calculate in Section~\ref{FPerS}, are related to so-called Zamolodchikov
periodicity conjecture.
They can be deduced from the simplest $A_k\times{}A_w$-case;
the periodicity was proved in this case in~\cite{Vol} by a different method.
The general case of the conjecture was recently proved in~\cite{Kel}. 
It would be interesting to investigate an approach based on linear difference equations
in this case.

Finally, the maps that we calculate in Section~\ref{SPerS},
correspond to self-dual difference equations.
They do not enter into the framework of Zamolodchikov
periodicity conjecture and seem to be new.

\subsection{Periodicity of $\SL_{k+1}$-frieze patterns and generalized Gauss maps}
\label{FPerS}

An important property of $\SL_{k+1}$-frieze patterns is their
periodicity.

\begin{corollary}
\label{periodic}
Tame $\SL_{k+1}$-frieze patterns of width $w$ are $n$-periodic in the horizontal direction: 
$d_{i,j}=d_{i+n,j+n}$ for all $i,j$, where $n=k+w+2$.
\end{corollary}

This statement is a simple corollary of Theorem~\ref{TriThm}, Part~(i).
Note that,
in the simplest case $k=1$, the above statement was proved by Coxeter~\cite{Cox}.

Let us introduce the notation:
$$
U(a_1,a_2,\dots,a_k):=\left|\begin{array}{cccccc}
a_{1}&1&&&\\[4pt]
1&a_{2}&1&&\\[4pt]
&\ddots&\ddots&\ddots&\\[4pt]
&&1&a_{k-1}&1\\[4pt]
&&&1&a_{k}
\end{array}\right|.
$$
for the simplest determinants from Section~\ref{DeTS},
see Example~\ref{HillEx2}, part (a).
We obtain the following family of rational periodic maps.

\begin{corollary} 
\label{main}
Let the rational map $F: \R^{n-3} \to \R^{n-3}$ be given by the formula
\begin{equation*}
\label{FMap}
F:(a_1,a_2,\dots,a_{n-3})\mapsto
\left(a_2,a_3,\dots,a_{n-3}, P(a_1,a_2,\dots,a_{n-3})\right),
\end{equation*}
where 
$$
P(a_1,a_2,\dots,a_{n-3}) = \frac{1+U(a_1,\dots,a_{n-4})}{U(a_1,\dots,a_{n-3})}.
$$
Then $F^n=\mathrm{id}$.
\end{corollary}

\begin{proof} 
Consider the periodic Hill equation $V_i=a_iV_{i-1}-V_{i-2}$,
or equivalently $\SL_2$-frieze  pattern whose first row consists of ones, 
and the second row is the bi-infinite sequence $(a_i)$. 
The entries of $k$th row of this frieze pattern are 
$$
U(a_i,\dots,a_{i+k-2}),\qquad i\in\Z,
$$
see Example~\ref{HillEx2} and \cite{Cox,MGOT}. 

Assume that all the solutions of the Hill equation are $n$-antiperiodic. 
Then the frieze pattern is closed of width $n-3$ and its rows are $n$-periodic, 
see \cite{Cox} and Corollary~\ref{periodic}.
Furthermore, the $(n-2)$th row of a closed $\SL_2$-frieze pattern
consists in $1$'s.
Therefore 
$$
U(a_i,a_{i+1},\dots,a_{i+n-3})=1,
$$
 for all $i$. 

On the other hand, decomposing the determinant $U(a_i,a_{i+1},\dots,a_{i+n-3})$
by the last row, we find
\begin{equation} 
\label{rec1}
a_{i+n-3}=P(a_i,\dots,a_{i+n-4}).
\end{equation}
 We can choose $a_1,\dots,a_{n-3}$ arbitrarily and then consecutively define 
 $a_{n-2}, a_{n-1},\dots$ using formula (\ref{rec1}) for $i=1,2,\dots$. 
 That is, we reconstruct the sequence $(a_i)$ from the ``seed" 
 $\{a_1,\dots,a_{n-3}\}$ by iterating the map $F$. 
 By periodicity assumption, the result is an $n$-periodic sequence.  
 \end{proof}

Introduce another notation:
$$
V(a_1,b_1,\dots,b_{k-1},a_k):=\left|
\begin{array}{llllll}
a_{1}&b_{1}&1&&&\\[4pt]
1&a_{2}&b_{2}&1&&\\[4pt]
&\;\ddots&\; \ddots& \;\ddots&\; \ddots&\\[4pt]
&&1&a_{k-2}&b_{k-2}&1\\[4pt]
&&&1&a_{k-1}&b_{k-1}\\[4pt]
&&&&1&a_k
\end{array}
\right|,
$$
see Example~\ref{HillEx2}, part (b).

\begin{corollary} 
\label{mainBis}
Let the rational map $\Phi: \R^{2n-8} \to \R^{2n-8}$ be given by the formula
\begin{equation*}
\label{PhiMap}
\Phi:(a_1,b_1,a_2,b_2,\dots,a_{n-4},b_{n-4})\mapsto\left(b_1,a_2,b_2,\dots,b_{n-4},Q(a_1,b_1,\dots,b_{n-4},a_{n-4})\right),
\end{equation*}
where
$$
Q(a_1,b_1,\dots,a_{n-4},b_{n-4})=
\frac{1+b_{n-4}V(a_1,b_1,\dots,a_{n-5})-V(a_1,b_1,\dots,a_{n-6})}
{V(a_1,b_1,\dots,a_{n-4})}.
$$
Then $\Phi^{2n}=\mathrm{id}$.
\end{corollary}

\begin{proof} 
The arguments are similar to those of the above proof,
but we will also use the notion of a projectively dual equation.

Consider the difference equation
$V_i=a_iV_{i-1}-b_iV_{i-2}+V_{i-3}$
and assume that all its solutions (and therefore all its coefficients) are $n$-periodic.
Consider the dual equation, see Example~\ref{DeEx}, 
but ``read'' it from right to left:
$$
V^*_{i-3}=a_iV^*_{i-2}-b_{i+1}V^*_{i-1}+V^*_i.
$$

The map $\Phi$ associates to a ``seed'' $\{a_1,b_1,\dots,a_{n-4},b_{n-4}\}$
of the initial equation the same ``seed'' of the dual equation.

Recall finally that the double iteration of the projective duality is a shift:
$i\to{}i+1$.
Therefore, $\Phi^2:(a_i,b_i,\ldots)\mapsto(a_{i+1},b_{i+1},\ldots)$,
which is $n$-periodic by assumption.
\end{proof}

\begin{example} \label{penta}
{\rm For $n=5$, the maps from Corollaries~\ref{main} and~\ref{mainBis} are as follows:
$$
F(a_1,a_2)=\left(a_2,\frac{1+a_1}{a_1a_2-1}\right), \qquad 
\Phi(b,a)=\left(a,\frac{a+1}{b}\right).
$$
The first one is the classical $5$-periodic Gauss map, the second, which looks even
more elementary, is $10$-periodic.
}
\end{example}

\subsection{Periodic maps in the self-dual case}\label{SPerS}

Let us now consider a version of the rational periodic maps
that correspond to self-dual third-order $n$-periodic equations,
see Section~\ref{SDS}.

\begin{corollary} 
\label{selfd}
(i) 
Let $n=2m-1$, and let the rational map $G_{o}: \R^{n-4} \to \R^{n-4}$ be given by the formula
$$
G_{o}(a_1,b_1,a_2,b_2,\dots,a_{m-2},b_{m-2})=\left(b_1,a_2,b_2,\dots,b_{m-2}, R_o(a_1,b_1,\dots,a_{m-2},b_{m-2})\right),
$$
where 
$$
\begin{array}{l}
R_o(a_1,b_1,\dots,a_{m-2},b_{m-2}) = \\[6pt]
\displaystyle
\qquad
\qquad
\qquad
\frac{V(a_2,b_2,\dots,a_{m-2}) + b_{m-2} V(a_1,b_1,\dots,a_{m-3})-V(a_1,b_1,\dots,a_{m-4})}{V(a_1,b_1,\dots,a_{m-2})}.
\end{array}
$$
Then $G_o^n=\mathrm{id}$.

(ii) Let $n=2m$, and let the rational map $G_{e}: \R^{n-4} \to \R^{n-4}$ be given by the formula
$$
G_e(a_1,b_1,a_2,b_2,\dots,a_{m-2},b_{m-2})=
\left(
b_1,a_2,b_2,\dots,b_{m-2}, R_e(a_1,b_1,\dots,a_{m-2},b_{m-2})\right),
$$
where 
$$
\begin{array}{l}
R_e(a_1,b_1,\dots,a_{m-2},b_{m-2}) = \\[6pt]
\displaystyle
\qquad
\qquad
\qquad
\frac{V(b_2,a_3,\dots,a_{m-2}) + b_{m-2} V(a_1,b_1,\dots,a_{m-3})-V(a_1,b_1,\dots,a_{m-4})}{V(a_1,b_1,\dots,a_{m-2})}.
\end{array}
$$
Then $G_e^n=\mathrm{id}$.\\
\end{corollary} 

\begin{proof}
Let us consider part (i); the other case is similar. 

As before, consider $n$-periodic equations $V_i=a_iV_{i-1}-b_iV_{i-2}+V_{i-3}$.
As before, both sequences $(a_i)$ and $(b_i)$ are $n$-periodic, so that
the sequence
$$
\dots,b_1,a_1,b_2,a_2,\dots,b_n,a_n,\dots
$$
is $2n$-periodic.
However, if the equation is self-dual, then this sequence is, 
actually, $n$-periodic.
More precisely,
$$
\left\{
\begin{array}{ll}
b_{i+m}=a_i, & n=2m-1;\\[4pt]
 b_{i+m}=b_i,\, a_{i+m}=a_i , & n=2m.
\end{array}
\right.
$$

We are concerned with the shift
$$
(a_1,b_1,a_2,b_2,\dots,a_{m-2},b_{m-2}) \mapsto 
(b_1,a_2,b_2,\dots,a_{m-2},b_{m-2},a_{m-1}),
$$ 
and we need to express $a_{m-1}$ as a function of $a_1,b_1,a_2,b_2,\dots,a_{m-2},b_{m-2}$.

To this end, we express, in two ways, the same entry of the 
$\SL_3$-frieze pattern. 
Consider the South-East diagonal through the entry $a_1$ of the top 
non-trivial row and the South-West diagonal through the entry $a_{m-1}$ of the same row. 
The entry at their intersection is $V(a_1,b_1,\dots,a_{m-1})$, see \cite{MGOT}. 

Since the difference equation is self-dual, 
the bottom non-trivial row of the $\SL_3$-frieze pattern is identical to the top one. 
The North-East diagonal through the entry $a_2$ of this bottom non-trivial row intersects the North-West diagonal through the entry $a_{m-2}$ of the same row at the same point as above, and the entry there is $V(a_2,b_2,\dots,a_{m-2})$. 

Therefore 
$$
V(a_1,b_1,\dots,a_{m-1})=V(a_2,b_2,\dots,a_{m-2}),
$$ 
and it remains to solve this equation for $a_{m-1}$. 
This yields the formula for the rational function~$R$.
\end{proof}

\begin{example}
{\rm If $n=6$, we obtain the map
$$
G_e(a,b)=\left(b,\frac{2b}{a}\right)
$$
that indeed has order 6. If $n=7$, we obtain the map
$$
G_o(a_1,b_1,a_2,b_2)=\left(b_1,a_2,b_2,\frac{a_2+a_1b_2-1}{a_1a_2-b_1}\right)
$$
that has order 7.
}
\end{example}

\subsection{Another expression for $2n$-periodic maps}

Let us sketch a derivation of the map $\Phi$ from a
geometrical point of view. The formula is:
$$
\Phi(x_1,\dots,x_{2n-8})=
\left(x_2,\dots,x_{2n-8}, R(x_1,\dots,x_{2n-8})\right),
$$
where
$$
R(x_1,\dots,x_{2n-8})=
\frac{O^{2n-7}_{-1}}{O^{2n-9}_{-1}-x_{2n-8}x_{2n-9}\,O^{2n-11}_{-1}},
$$
and where $O_a^b$ is defined by recurrence relation
$$
O_a^b=O_a^{b-2}-x_{b-2}O_a^{b-4}+x_{b-2}x_{b-3}x_{b-4}O_a^{b-6},
\qquad
a=b-4,b-6,\ldots
$$
with the initial conditions $O_b^{b}=O_{b-2}^{b}=1$.
We will see below why $\Phi^{2n}$ is the identity.

\begin{example}
The first non-trivial example is
$$
\Phi(x_1,x_2)=\left(
x_2,\frac{1-x_1}{1-x_1x_2}
\right),
$$
which is the Gauss map.
The next example is:
$$
\Phi(x_1,x_2,x_3,x_4)=\left(
x_2,x_3,x_4,\frac{1-x_1-x_3+x_1x_2x_3}{1-x_1-x_3x_4}
\right),
$$
which is $12$-periodic.
\end{example}

These formulas give the expression of the map $\Phi$
in so-called corner coordinates on the moduli space of $n$-gons in $\PP^2$.
Let $LL'$ denotes the intersection of
lines $L$ and $L'$ and let 
$PP'$ denote the line containing $P$ and $P'$.
We have the {\it inverse cross-ratio\/}
\begin{equation}
\label{inverse-cross-ratio}
[a,b,c,d]=\frac{(a-b)(c-d)}{(a-c)(b-d)}.
\end{equation}

A {\it polygonal ray\/} is an infinite collection of
points $P_{-7},P_{-3},P_{+1},...$, with
indices congruent to $1$ mod $4$, normalized so that
\begin{equation*}
P_{-7}=(0,0,1), \hskip 15 pt
P_{-3}=(1,0,1), \hskip 15 pt
P_{+1}=(1,1,1), \hskip 15 pt
P_{+5}=(0,1,1).
\end{equation*}
These points determine lines
\begin{equation*}
L_{-5+k}=P_{-7+k}P_{-3+k}, \hskip 30 pt
\end{equation*}
and also the flags
\begin{equation*}
F_{-6+k}=(P_{-7+k},L_{-5+k}), \hskip 15 pt
F_{-4+k}=(P_{-3+k},L_{-5+k})
\end{equation*}

We associate {\it corner invariants\/} to the flags, as follows.
\begin{equation*}
c(F_{0+k})=[P_{-7+k},P_{-3+k},L_{-5+k}L_{3+k},L_{-5+k}L_{7+k}],
\end{equation*}
\begin{equation*}
c(F_{2+k})=[P_{9+k},P_{5+k},L_{7+k}L_{-1+k},L_{7+k}L_{-5+k}],
\end{equation*}
All these equations are meant for $k=0,4,8,12,....$
Finally, we define

\begin{equation*}
x_k=c(F_{2k}); \hskip 30 pt k=0,1,2,3...
\end{equation*}
The quantities
$x_0,x_1,x_2,...$ are known as the corner invariants of the ray.
\newline
\newline
{\bf Remark:\/} From the exposition here,
it would seem  more natural to call these
invariants {\it flag invariants\/}, though in the past we
have called them {\it corner invariants\/}.
\newline

We would like to go in the other direction:
Given a list $(x_0,x_1,x_2,...)$,
 we seek a polygonal ray which
has this list as its flag invariants.
Taking \cite{Sch1}, eq. (20) and applying a suitable
projective duality, we get the {\it reconstruction formula\/}

\begin{equation*}
\label{reconstruct}
P_{9+2k}=
\begin{bmatrix}1&-1&x_0x_1 \\ 1&0&0 \\ 1&0&x_0x_1\end{bmatrix}
\begin{bmatrix}O^{3+k}_{-1} \\ O^{3+k}_{+1} \\ O^{3+k}_{+3}\end{bmatrix}
\hskip 30 pt
k=0,2,4,...
\end{equation*}

Multiplying through by the matrix $M^{-1}$, where $M$ is the
matrix in equation \eqref{reconstruct}, we get an alternate
normalization.  Setting
\begin{equation*}
\label{norm11}
Q_{-7}=\begin{bmatrix}0\\ x_0x_1\\ 1\end{bmatrix},\hskip 15 pt
Q_{-3}=\begin{bmatrix}0\\ 0\\ x_0x_1\end{bmatrix},\hskip 15 pt
Q_{1}=\begin{bmatrix}1 \\ 0 \\ 0\end{bmatrix},\hskip 15 pt
Q_{5}=\begin{bmatrix}1 \\ 1 \\ 0\end{bmatrix}
\end{equation*} 
we have
\begin{equation*}
\label{norm12}
Q_{9+2k}=\begin{bmatrix}O^{3+k}_{-1} \\ O^{3+k}_{+1} \\ O^{3+k}_{+3}\end{bmatrix}
\hskip 30 pt
k=0,2,4,...
\end{equation*}

In case we have a closed $n$-gon, we have
\begin{equation}
\label{repeat}
\begin{bmatrix}0\\0\\1\end{bmatrix}=[Q_{-3}]=[Q_{4n-3}]=[Q_{9+2(2n-6)}]=
\begin{bmatrix}O^{2n-3}_{-1}\\ O^{2n-3}_{+1} \\ O^{2n-3}_{+3}\end{bmatrix}.
\end{equation}
Here $[\cdot ]$ denotes the equivalence class in the projective
plane.
Equation \eqref{repeat} yields
$O^{2n-3}_{-1}=O^{2n-3}_{+1}=0$.
Shifting the vertex labels of our polygon by $1$ unit
has the effect of shifting the flag invariants
by $2$ units.  Doing all cyclic shifts, we get
\begin{equation}
\label{rel1}
O_a^b=0 \hskip 30 pt b-a=2n-4,2n-2, \hskip 30pt
a,b\ {\rm odd\/}.
\end{equation}
Given a polygon $P$ with flag
invariants $x_1,x_2,...$ we consider
the dual polygon $P^*$.  The polygon $P^*$ is such that
a projective duality carries the lines extending the
edges of $P^*$ to the points of $P$, and {\it vice versa\/}.
When suitably labeled, the flag invariants of $P^*$ are
$x_2,x_3,\ldots$. For this reason, equation \eqref{rel1} also
holds when both $a$ and $b$ are even.  In particular,
we have the $2n$ relations:
\begin{equation}
\label{rel2}
O_a^b=0, \hskip 30 pt b-a=2n-4.
\end{equation}
Equation \eqref{rel2} tells us that
$O_{-1}^{2n-5}=0$.  But now our basic recurrence relation gives
\begin{equation*}
x_{2n-7}x_{2n-8}x_{2n-9}O_{-1}^{2n-11}-x_{2n-7}O_{-1}^{2n-9}+O_{-1}^{2n-7}=0.
\end{equation*}
Note that $x_0$ does not occur in this equation.
Solving for $x_{2n-7}$, we get
$$
x_{2n-7}=R_n(x_1,...,x_{2n-8}),$$
where $R_n$ is the expression that occurs in the map $\Phi$ above.
Thanks to equation \eqref{rel2},
these equations hold when we shift the indices cyclically by
any amount.  Thus
\begin{equation*}
x_{2n-7+k}=R_n(x_{k+1},...,x_{2n-8+k}), \hskip 30 pt k=1,...,2n.
\end{equation*}
This is an explanation of why $\Phi^{2n}$ is the identity.

\section{Relation between the spaces $\Ec_{w+1,n},\Fc_{w+1,n}$ and
 $\Cc_{k+1,n}$}\label{TriThmS}

In this section, we give more details about the relations
between the main spaces studied in this paper
and complete the proof of Theorem \ref{TriThm}.

\subsection{Proof of Theorem \ref{TriThm}, Part (i)}\label{TriThmP1S}
Let us prove that the map 
${\mathcal  E}_{k+1,n}\longrightarrow{\mathcal  F}_{k+1,n}$
constructed in Section~\ref{1stIso} is indeed an isomorphism
of algebraic varieties.

{\bf A}.
Let us first prove that this map is a bijection.
By the (anti)periodicity assumption (\ref{APeriod}), after $n-k-2$ numbers on each North-East diagonal, there appears 1, followed by $k$ zeros. Thus we indeed obtain an array of numbers bounded by a row of ones and $k$ rows of zeros.

Consider the determinants $D_{i,j}$ defined by (\ref{DEq})
for $j\geq{}i-1$. We show that $D_{i,j}=1$ by induction on $j$, assuming $i$ is fixed.
By construction, $k+1$ consecutive North-East diagonals of the frieze give a sequence $(V_j)_{j\in \Z}$
of vectors in $\R^{k+1}$ satisfying the difference equation \eqref{REq}, where
$$
V_j=
\begin{pmatrix}
d_{i,j}\\
d_{i+1,j}\\
\vdots\\
d_{i+k,j}
\end{pmatrix}.
$$
One has $D_{i,j}=\left|V_j,V_{j+1}, \cdots, V_{j+k}\right|$. 
Now we compute the first determinant
$$
D_{i,i-1}=
\left\vert
\begin{array}{cccc}
1&d_{i,i}&\ldots&d_{i,i+k-1}\\[4pt]
0&1&\ldots&d_{i+1,i+k-1}\\[4pt]
\ldots& \ldots&& \ldots\\[4pt]
0&0&\ldots&1
\end{array}
\right\vert=1.
$$
Since  the last coefficient in the equation \eqref{REq} satisfied by the sequence 
of the vectors $(V_j)$ is $a_{i}^{k+1}=(-1)^{k}$, 
one easily sees that $D_{i,j}=D_{i,j+1}$. 
Therefore $D_{i,j+1}=D_{i,j}=\cdots=1$. We have proved that the array 
$d_{i,j}$  is an $\SL_{k+1}$-frieze pattern.

\begin{lemma}
The defined frieze pattern is tame. 
\end{lemma}

\noindent
\begin{proof}
The space of solutions of the difference equation (\ref{REq}) is $k+1$-dimensional, therefore  every 
$(k+2)\times (k+2)$-block cut from $k+2$ consecutive diagonals
gives a sequence of linearly dependent vectors. Hence every $(k+2)\times (k+2)$-determinant vanishes.
\end{proof}

Conversely, consider a tame $\SL_{k+1}$-frieze pattern. 
Let $\eta_j=(\ldots,d_{i,j},d_{i+1,j},\ldots)$ be
the $j$th South-East diagonal. We claim that, for every $j$, the diagonal $\eta_{j}$ is a linear combination of $\eta_{j-1},\dots,\eta_{j-k-1}$:
\begin{equation}
\label{etarel}
\eta_{j}=a_{j}^{1} \eta_{j-1} - a_{j}^{2} \eta_{j-2} + \cdots + (-1)^{k-1} a_{j}^{k} \eta_{j-k}+(-1)^{k}a_{j}^{k+1} \eta_{j-k-1}.
\end{equation}
Indeed, consider a $(k+2)\times (k+2)$-determinant whose last column is on the diagonal $\eta_{j}$. Since the determinant vanishes, the last column is a linear combination of the previous $k+1$ columns. To extend this linear relation to the whole diagonal, slide the $(k+2)\times (k+2)$-determinant in the $\eta$-direction; this yields (\ref{etarel}). 

Next, we claim that $a_{j}^{k+1}=1$. To see this, choose a $(k+1)\times (k+1)$-determinant whose last column is on the diagonal $\eta_{j}$. By definition of $\SL_{k+1}$-frieze pattern, this determinant equals~1. On the other hand, due to the relation (\ref{etarel}), this determinant equals $a_{j}^{k+1}$ times a similar $(k+1)\times (k+1)$-determinant whose last column is on the diagonal $\eta_{j-1}$. The latter also equals~1, therefore $a_{j}^{k+1}=1$. 

We have shown that each North-East diagonal of the $\SL_{k+1}$-frieze pattern 
consists of solutions of the linear difference equation (\ref{etarel}) with $a_{j}^{k+1}=1$, 
that is, the difference equation (\ref{REq}). 
By definition of an $\SL_{k+1}$-frieze pattern, these solutions are (anti)periodic. Hence the coefficients are periodic as well. 

We proved that the map 
${\mathcal  E}_{k+1,n}\to{\mathcal  F}_{k+1,n}$
constructed in Section~\ref{1stIso} is one-to-one.

{\bf B}.
Let us now show that this map is a morphism of
algebraic varieties defined in Sections~\ref{AlgVarS} and~\ref{SubSGr}.

Recall that the structure of algebraic variety on $\Ec_{k+1,n}$
is defined by polynomial equations on the coefficients resulting from
the (anti)periodicity.
More precisely, these relations can be written in the form:
$$
d_{i,i+w}=1,
\qquad
d_{i,j}=0, \quad w<j<n,
$$
where $d_{i,j}$ are defined by~(\ref{ConstEq})
and calculated according to the formulas (\ref{DetEq}) and (\ref{DetEqDva})
that also make sense for~$j\geq{}w$.
These polynomial equations guarantee that the solutions of the equation~(\ref{REq}) are
$n$-(anti)periodic.
In other words, if $M$ is the monodromy
operator of the equation (which is, as well-known, an element of  the group $\SL_{k+1}$),
then the (anti)periodicity condition means that
$M=(-1)^k\mathrm{Id}$.
Note that exactly $k(k+2)$ of these equations are
algebraically independent, since this is the dimension of $\SL_{k+1}$.

The structure of algebraic variety on
the space $\Fc_{k+1,n}$ is given by the embedding
into $\Gr_{k+1,n}$.
The Grassmannian itself is an algebraic variety
defined by the Pl\"ucker relations, and the embedding $\Fc_{k+1,n}\subset\Gr_{k+1,n}$
is defined by the conditions that some of the Pl\"ucker coordinates are equal
to each other.

We claim that the map ${\mathcal  E}_{k+1,n}\to{\mathcal  F}_{k+1,n}$,
constructed in Section~\ref{1stIso},
is a morphism of algebraic varieties (i.e., a birational map).
Indeed, the coefficients $a_i^j$ of the equation~(\ref{REq})
are pull-backs of rational functions in Pl\"ucker coordinates,
see formula~\eqref{DuDeT}.
Moreover, all the determinants in these formulas are equal to $1$
when restricted to $\Ec_{k+1,n}$.
Conversely, the Pl\"ucker coordinates restricted to $\Fc_{k+1,n}$
are polynomial functions in~$a_i^j$.
This follows from the formulas~\eqref{DetEq} and~\eqref{DetEqDva}
and from the fact that the Pl\"ucker coordinates are polynomial in $d_{i,j}$.

This completes the proof of Theorem~\ref{TriThm}, Part (i).

\subsection{Second isomorphism, when $n$ and $k+1$ are coprime}\label{CoPSeC}

Let us prove that
the spaces of difference equations (\ref{REq}) with (anti)periodic solutions 
and the moduli space of $n$-gons in $\RP^k$ are  isomorphic algebraic varieties, 
provided the period $n$ and the dimension $k+1$ have no common divisors.

{\bf A}.
We need to check that the map (\ref{MaPP}) is, indeed, a one-to-one correspondence
between ${\mathcal E}_{k+1,n}$ and ${\mathcal C}_{k+1,n}$. 
Let us construct the inverse map to (\ref{MaPP}).
Consider a non-degenerate $n$-gon $(v_i)$ in~$\RP^k$.
Choose an arbitrary lift $(\tilde V_i)\in\R^{k+1}$ of the vertices.
The $k+1$  coordinates $\tilde V_i^{(1)},\ldots,\tilde  V_i^{(k+1)}$
of the vertices of this $n$-gon are solutions to some
(and the same) difference equation (\ref{REq}) if and only if the determinant (\ref{detconst})
is constant (i.e., independent of $i$).
We thus wish to define a new lift $V_i=t_i \tilde V_i$ such that
$$
t_i t_{i+1}\cdots t_{i+k}
\left|
\tilde V_i, \tilde V_{i+1},\ldots,\tilde V_{i+k}
\right|=1.
$$
This system of equations on $t_1,\ldots, t_n$ has a unique solution if and only if $n$ and $k+1$ are coprime.
Finally, two projectively equivalent $n$-gons correspond to the same equation. 
Thus the map (\ref{MaPP}) is a bijection.

{\bf B}.
The structures of algebraic varieties are in full accordance
since the projection from~$\Gr_{k+1,n}$ to ${\mathcal C}_{k+1,n}$
is given by the projection with respect to
the $\mathbb{T}^{n-1}$-action which is an algebraic action of an algebraic group.

Theorem~\ref{TriThm}, Part (ii) is proved.

\subsection{Proof of Proposition~\ref{NewProp}}\label{ProProSec}

Consider finally the case where $n$ and $k+1$ have common divisors.
Suppose that $\gcd(n,k+1)=q \neq 1$. 
In this case, the constructed map 
is not injective and its image is a subvariety in the moduli space of polygons.

As before, we assign an $n$-gon $V_1,\ldots,V_n \in \R^{k+1}$ to a difference equation. 
Given numbers $t_0,\ldots,t_{q-1}$ whose product is 1, we can rescale
$$
V_j \mapsto t_{j\ {\rm mod}\ q}\ V_j,\quad j=1,\ldots,n,
$$
keeping the determinants $\left|V_i, V_{i+1},\ldots,V_{i+k}\right|$ intact.
This action of $({\R^*})^{q-1}$ does not affect the projection of the polygons to $\RP^k$. Thus this projection has at least $q-1$-dimensional fiber.

To find the dimension of the fiber, we need to consider the system of equations
$$
t_i t_{i+1}\cdots t_{i+k}=1,\quad i=1,\ldots,n,
$$
where, as usual, the indices are understood cyclically mod $n$. Taking logarithms, this is equivalent to a linear system with the circulant matrix
$$
\left( 
\begin{array}{ccccccccc}
1&1&\ldots&1&1&0&\ldots&0&0\\
0&1&\ldots&1&1&1&0&\ldots&0\\[4pt]
&&\ddots&&&&\ddots\\[4pt]
0&0&\ldots&0&1&1&1&\ldots&1\\
1&0&\ldots&0&0&1&1&\ldots&1\\[4pt]
&&\ddots&&&&\ddots\\[4pt]
1&1&\ldots&1&0&0&\ldots&0&1\\
\end{array}
\right)
$$
with $k+1$ ones in each row and column. 

The eigenvalues of such a matrix are given by the formula
$$
1+\omega_j+\omega_j^2 + \dots + \omega_j^k,\quad j=0,1,\ldots,n-1,
$$
where $\omega_j = \exp(2\pi {\bf i} j/n)$ is $n$th root of 1, see \cite{Da}. If $j>0$, the latter sum equals 
$$
\frac{\omega_j^{k+1}-1}{\omega_j -1},
$$
and it equals 0 if and only if $j(k+1) = 0 \mod n$. This equation has $q-1$ solutions, hence the circulant matrix has corank $q-1$.
Proposition~\ref{NewProp} is proved.

\subsection{The $\SL_2$-case: relations to Teichm\"uller theory}

Now we give a more geometric description of the image of the map~(\ref{MaPP})
in the case $k=1$.
If $n$ is odd then this map is one-to-one, but if $n$ is even then 
its image has codimension 1. 
To describe this image, we need some basic facts from decorated Teichm\"uller theory \cite{Pen,Pen2}.

Consider $\RP^1$ as the circle at infinity of the hyperbolic plane. 
Then a polygon in $\RP^1$ can be thought of as an ideal polygon in
the hyperbolic plane $\mathcal{H}^2$. 
A decoration of an ideal $n$-gon is a choice of horocycles centered at its vertices. 

Choose a decoration and define the side length of the polygon as the signed hyperbolic distance 
between the intersection points of the respective horocycles with this side; the convention is that  
if the two consecutive horocycles are disjoint then the respective distance is positive
(one can always assume that the horocycles are ``small" enough). 
Denoting the side length by~$\delta$, the lambda length is defined as $\lambda=\exp{(\delta/2)}$.

Let $n$ be even. Define the alternating perimeter length of an ideal $n$-gon: choose a decoration and consider the alternating sum of the side lengths.
The alternating perimeter length of an ideal even-gon does not depend on the decoration: changing a horocycle adds (or subtracts) the same length to two adjacent sides of the polygon and does not change the alternating sum. 

\begin{proposition}
\label{perim}
The image of the map~(\ref{MaPP}) with $k=1$ and $n$ even
consists of polygons with zero alternating perimeter length.
\end{proposition}

\begin{proof}
Let $(v_i)$ be a polygon in $\RP^1$. Let $x_i=[v_{i-1},v_i,v_{i+1},v_{i+2}]$ be the cross-ratio of the four consecutive vertices. 
Of six possible definitions of cross-ratio, we use the following one:
$$
[t_1,t_2,t_3,t_4]=\frac{(t_1-t_3)(t_2-t_4)}{(t_1-t_2)(t_3-t_4)}.
$$
This is the reciprocal of the formula in equation \eqref{inverse-cross-ratio}.

We claim that a $2n$-gon is in the image of~(\ref{MaPP}) if and only if 
\begin{equation}
\label{altprod}
\prod_{i\ {\rm odd}} x_i = \prod_{i\ {\rm even}} x_i.
\end{equation}

Indeed, let $(V_i)\subset \R^2$ be an anti periodic solution to the discrete Hill's equation
$$
V_{i+1}=c_i V_i - V_{i-1}
$$
with $|V_i,V_{i+1}|=1$. Then 
$$
x_i=\frac{|V_{i-1},V_{i+1}| |V_i,V_{i+2}|}{|V_{i-1},V_i| |V_{i+1},V_{i+2}|} = c_i c_{i+1}.
$$
Therefore (\ref{altprod}) holds.

Conversely, let (\ref{altprod}) hold. Let
$(\tilde V_i)$ be a lift of $(v_i)$ to $\R^2$ with $|\tilde V_i,\tilde V_{i+1}|>0$. 
As before, we want to renormalize these vectors so that the consecutive determinants equal 1. 
This boils down to solving the system of equations $t_i t_{i+1} |\tilde V_i,\tilde V_{i+1}|=1$. This system has a solution if and only if
$$
\prod_{i\ {\rm odd}} |\tilde V_i,\tilde V_{i+1}|= \prod_{i\ {\rm even}} |\tilde V_i,\tilde V_{i+1}|.
$$
On the other hand, one computes that
$$
1=\frac{\prod_{i\ {\rm odd}} x_i}{\prod_{i\ {\rm even}} x_i} = \left(\frac{\prod_{i\ {\rm even}} |\tilde V_i,\tilde V_{i+1}|}{\prod_{i\ {\rm odd}} |\tilde V_i,\tilde V_{i+1}|}\right)^2,
$$
and the desired rescaling exists.

Finally, one relates cross-ratios with lambda lengths, see \cite{Pen,Pen2}:
$$
[v_{i-1},v_i,v_{i+1},v_{i+2}]=\frac{\lambda_{i-1,i+1} \lambda_{1,i+2}}{\lambda_{i-1,i}\lambda_{i+1,i+2}}.
$$
Therefore
$$
1=\frac{\prod_{i\ {\rm odd}} x_i}{\prod_{i\ {\rm even}} x_i} = \left(\frac{\prod_{i\ {\rm even}} \lambda_{i,i+1}}{\prod_{i\ {\rm odd}} \lambda_{i,i+1}}\right)^2 = e^{\sum (-1)^i \delta_i}.
$$
It follows that the alternating perimeter length is zero.
\end{proof}

\medskip

\noindent \textbf{Acknowledgments}.
We are pleased to thank D. Leites and A.~Veselov for enlightening discussions.
This project originated at the Institut Math\'ematique de Jussieu (IMJ), Universit\'e Paris~6. 
S.~T. is grateful to IMJ for its hospitality.
S.~M-G. and V.~O. were partially supported by the PICS05974 ``PENTAFRIZ'' of CNRS. 
R.E.S. was supported by NSF grant DMS-1204471.
S.~T. was  supported by   NSF grant DMS-1105442.

\end{document}